\documentclass[11pt]{article}

\usepackage{amssymb,amsmath,bm}
\usepackage{amsthm}

\usepackage{textcomp}
\usepackage{enumerate}      
\usepackage{graphicx}        
\usepackage{caption}

\usepackage{mathrsfs}

\usepackage{url}

\renewcommand{\setminus}{{\smallsetminus}}

\newcommand{\cp}[1]{\vcenter{\hbox{#1}}}

\newtheorem*{namedtheorem}{\theoremname}
\newcommand{\theoremname}{testing}
\newenvironment{named}[1]{\renewcommand{\theoremname}{#1}\begin{namedtheorem}}{\end{namedtheorem}}

\newtheorem{theorem}{Theorem}[section]
\newtheorem{lemma}[theorem]{Lemma}
\newtheorem{proposition}[theorem]{Proposition}
\newtheorem{definition}[theorem]{Definition}
\newtheorem{corollary}[theorem]{Corollary}
\newtheorem{conjecture}[theorem]{Conjecture}
\newtheorem{question}[theorem]{Question}

\theoremstyle{remark}
\newtheorem{remark}[theorem]{Remark}

\theoremstyle{remark}

\numberwithin{equation}{section}

\makeatother
\DeclareMathOperator{\Z}{\mathbb{Z}}
\DeclareMathOperator{\lb}{\lbrace}
\DeclareMathOperator{\rb}{\rbrace}
\DeclareMathOperator{\vol}{\mathrm{Vol}}

\begin{document}
\title{Turaev-Viro invariants, colored Jones polynomials and volume} 
\author{Renaud Detcherry, Efstratia Kalfagianni \footnote{E.K. is supported by NSF Grants DMS-1404754 and DMS-1708249}  \ \  and Tian Yang \footnote{T.Y. is supported by NSF Grant  DMS-1405066}} \date{}
\maketitle

\begin{abstract} 
We obtain a formula for the Turaev-Viro invariants of a link  complement in terms of values of the colored Jones polynomials of the link.  
As an application,  we  give  the first examples of 3-manifolds where the ``large $r$" asymptotics   of the Turaev-Viro invariants determine the hyperbolic volume.
We verify  the volume conjecture of  Chen and the third named author\,\cite{Chen-Yang}  for the figure-eight knot and the Borromean rings. Our calculations also exhibit new phenomena of asymptotic behavior of values of the colored Jones polynomials that seem  to be predicted neither by the Kashaev-Murakami-Murakami  volume conjecture and its generalizations nor by  Zagier's  quantum modularity conjecture. We conjecture that the asymptotics of the Turaev-Viro invariants of any link complement determine the simplicial volume of the link, and verify this conjecture for all knots with zero simplicial volume. Finally, we  observe that our simplicial volume conjecture is compatible with connected summations  and split unions of links.
\end{abstract}

\section{Introduction} In \cite{TuraevViro}, Turaev and Viro defined a family  of 3-manifold invariants as state sums on triangulations of  manifolds. 
The family is indexed by an integer $r,$ and for each $r$ the invariant depends on a choice of a  $2r$-th root of unity.
 In the last couple of decades these invariants have  been refined  and generalized in many directions and shown to be closely related to the Witten-Reshetikhin-Turaev invariants.
(See  \cite{BePe, KauLins, Turaevbook, Matveevbook} and references therein.) 
Despite these efforts, the relationship between the Turaev-Viro invariants and the geometric structures on $3$-manifolds arising from Thurston's geometrization  picture is not understood.
Recently, Chen and the third named author\,\cite{Chen-Yang} conjectured that, evaluated at appropriate roots of unity, the large-$r$ asymptotics  of the Turaev-Viro invariants of a complete hyperbolic $3$-manifold, with finite volume,
determine the hyperbolic volume of the manifold, and presented compelling experimental evidence to their conjecture.    

In the present paper, we focus mostly on the Turaev-Viro invariants of link complements in $S^3.$  Our main result  gives a formula of the Turaev-Viro invariants of a link  complement in terms of values of the colored Jones polynomials of the link. Using this formula we rigorously verify the volume conjecture of \cite{Chen-Yang} for the figure-eight knot and Borromean rings complement. To the best of our knowledge these are first examples of this kind.
Our calculations exhibit  new phenomena of asymptotic behavior  of the colored Jones polynomials that does not seem to be predicted by the  volume conjectures\,\cite{kashaev:vol-conj, murakami:vol-conj, gukov:vc} or by Zagier's quantum modularity conjecture\,\cite{Zagier}.


\subsection{Relationship between knot invariants}To state our results we need  to introduce some notation.
 For a link $L\subset S^3,$ let  $TV_r(S^3\setminus L,q)$ denote the $r$-th Turaev-Viro invariant of the link complement evaluated at a root of unity $q$ such that $q^2$ is primitive of degree $r.$  Throughout this paper, we will consider the case that $q=A^2,$ where $A$ is either a primitive $4r$-th root for any integer $r$ or  a primitive $2r$-th root for any odd integer $r.$

  We use the notation $\mathbf{i}=(i_1,\ldots,i_n)$ for a multi-integer of $n$ components (an $n$-tuple of integers) and use the notation $1\leqslant \mathbf{i}\leqslant m$
  to describe all such multi-integers with $1\leqslant i_k \leqslant m$ for each $k\in\{1,\dots, n\}.$
  Given a link $L$ with $n$ components, 
  let  $J_{L,\mathbf{i}}(t)$ denote the $\mathbf i$-th colored Jones polynomial of $L$ whose  $k$-th component is colored by $i_k$\,\cite{lickorish:book, KirbyMelvin}. If all the components of $L$ are colored by the same integer $i,$ then we simply denote $J_{L,(i,\dots,i)}(t)$ by $J_{L,i}(t).$ If $L$ is a knot, then $J_{L,i}(t)$ is the usual $i$-th colored Jones polynomial. The polynomials are
  indexed so that 
  $J_{L,1}(t)=1$ and  $J_{L, 2}(t)$ is the ordinary Jones polynomial, and are normalized so that 

  $$J_{U,i}(t)=[i] =\frac{A^{2i}-A^{-2i}}{A^2-A^{-2}}$$
 for the unknot $U,$ where by convention $t=A^4.$  Finally, we define
$$\eta_{r}=\frac{A^2-A^{-2}}{\sqrt{-2r}} \  \   {\rm and}  \  \
\eta_{r}'=\frac{A^2-A^{-2}}{\sqrt{-r}}.$$

Before stating our main result, let us recall once again the convention that $q=A^2$ and $t=A^4.$

\begin{theorem}\label{TV=CJP}
Let $L$ be a link in $S^3$ with $n$ components. 
\begin{enumerate}[(1)]
\item
For an integer $r\geqslant 3$  and a primitive $4r$-th root of unity $A,$ we have
$$TV_r(S^3\setminus L,q)=\eta_{r}^{2}\underset{1 \leqslant \mathbf{i}\leqslant r-1}{\sum} 
|  J_{L, \mathbf{i}} (t)   |^2.$$ 
\item
For an odd integer  $r\geqslant 3$ and a primitive $2r$-th root of unity $A,$ we have
$$TV_r(S^3\setminus L,q)=2^{n-1}(\eta_{r}')^{2}\underset{1 \leqslant \mathbf{i}\leqslant \frac{r-1}{2}}{\sum}
|  J_{L, \mathbf{i}} (t)|^2.$$ 
\end{enumerate}
\end{theorem}

Extending an earlier  result of Roberts\,\cite{Roberts}, Benedetti  and Petronio\,\cite{BePe}  showed that the invariants $TV_r(M,e^{\frac{\pi i}{r}})$ of a 3-manifold $M$, with non-empty  boundary, coincide up to a scalar with the $SU(2)$  Witten-Reshetikhin-Turaev  invariants of the double of $M.$  The first step in our proof of 
 Theorem \ref{TV=CJP} is to extend  this relation to the Turaev-Viro invariants and the $SO(3)$ Reshetikhin-Turaev invariants\,\cite{KirbyMelvin, BHMV1, BHMV2}. See Theorem \ref{RTdoubleTV}. 
For this we adapt the  argument of \cite{BePe}  to the case that $r$ is odd and $A$ is a primitive $2r$-th root of unity.
 Having this extension at hand,  the proof is completed by using the  properties of the
$SO(3)$  Reshetikhin-Turaev   Topological Qantum Field Theory (TQFT) developed by
Blanchet, Habegger, Masbaum and Vogel\,\cite{BHMV, BHMV2}.

Note that for any primitive $r$-th root of unity with $r\geqslant 3,$ the quantities $\eta_{r}$ and $\eta_{r}'$ are real and non-zero.
Since
$J_{L,1}(t)=1,$ and with the notation as in Theorem \ref{TV=CJP}, we have the following.

\begin{corollary}\label{positivity} For any $r\geqslant 3,$ any root $q=A^2$ and any link $L$ in $S^3,$ 
we have  $$TV_r(S^3\setminus L,q) \geqslant H_r>0,$$
 where $H_r=\eta_{r}^{2}$ in case (1), and  $H_r= 2^{n-1}(\eta_{r}')^{2}$ in case (2).
\end{corollary} 

Corollary \ref{positivity} implies that the invariants $TV_r(q)$ do not vanish for any link in $S^3$.
In contrast to that,  the values of the colored Jones polynomials  involved in the Kashaev-Murakami-Murakami volume conjecture\,\cite{kashaev:vol-conj, murakami:vol-conj} are known to vanish for   split links and for a class of links called Whitehead chains \cite{murakami:vol-conj, Rolandchains}. 
 
Another immediate consequence of Theorem  \ref{TV=CJP} is that links with the same colored Jones polynomials have the same Turaev-Viro invariants. In particular, since the colored Jones polynomials are invariant under Conway mutations and the genus 
2 mutations\,\cite{CJPmutation}, we obtain the following.

\begin{corollary} For any $r\geqslant 3,$ any root $q=A^2$ and any link $L$ in $S^3,$ the invariants $TV_r(S^3\setminus L,q)$ remain unchanged under Conway mutations and the genus 2 mutations.
\end{corollary}


\subsection{ Asymptotics of Turaev-Viro and colored Jones link invariants}

We are interested in the large $r$ asymptotics of the invariants $TV_r(S^3\setminus L,A^2)$ in the case that either 
$A=e^{\frac{\pi i}{2r }}$ for integers $r\geqslant 3$, or $A=e^{\frac{\pi i}{r}}$ for odd integers $r\geqslant 3.$
With these choices of $A,$ we have in the former case that 
$$\eta_{r}=\frac{2\sin(\frac{\pi}{r})}{\sqrt{2r}},$$ and in the latter case that
$$\eta_{r}'=\frac{2\sin(\frac{2\pi}{r})}{\sqrt{r}}.$$

In \cite{Chen-Yang}, Chen and the third named author presented experimental evidence and stated the following.

\begin{conjecture}\label{TVvolumeconj}\cite{Chen-Yang} For any 3-manifold $M$ with a complete hyperbolic structure of finite volume, we have
$$\lim_{r\to \infty} \frac {2\pi}  {r} \log (TV_r(M, e^{\frac{2\pi i}{r}}))=Vol(M),$$
where $r$ runs over all odd integers. 
\end{conjecture}

Conjecture \ref{TVvolumeconj} impies that $TV_r(M,e^{\frac{2\pi i}{r}})$ grows exponentially in $r.$
This is particularly surprising since the corresponding growth of $TV_r(M, e^{\frac{\pi i}{r}})$ is expected, and in many cases known, to be polynomial by Witten's asymptotic expansion conjecture\,\cite{witten, Jeffrey}. For closed 3-manifolds, this polynomial growth was established by  Garoufalidis\,\cite{Garoufalidis}. Combining \cite[Theorem 2.2]{Garoufalidis}
and the results of \cite{BePe}, one has  that for every 3-manifold $M$ with non-empty boundary, there exist  constants $C>0$ and $N$ such that $ |TV_r(M, e^{\frac{\pi i}{r}})|\leqslant C r^N.$
This together with Theorem \ref{TV=CJP}(1) imply the following.

\begin{corollary}\label{polynomial} For any link $L$ in $S^3,$ there exist constants $C>0$ and $N$ such that for any integer $r$ and  multi-integer $\mathbf i$ with $1\leqslant \mathbf{i}\leqslant r-1,$ the value of the $\mathbf i$-th colored Jones polynomial at $t=e^{\frac{2\pi i}{r}}$ satisfies 
$$|J_{L, \mathbf{i}}(e^{\frac{2\pi i}{r} })|\leqslant Cr^N.$$  Hence, $J_{L, \mathbf{i}}(e^{\frac{2\pi i}{r} })$ grows at most polynomially in $r.$
 \end{corollary}

As a main application of Theorem \ref{TV=CJP}, we provide the first rigorous evidence to Conjecture \ref{TVvolumeconj}.
 
\begin{theorem} \label{conjecturever} Let $L$ be either the figure-eight knot or the Borromean rings, and   let $M$ be the complement of $L$ in $S^3.$
Then
$$\lim_{r\rightarrow +\infty}\frac{2\pi}{r}\log TV_r(M, e^{\frac{2\pi i}{r}})=\lim_{m\rightarrow +\infty}\frac{4\pi}{2m+1}\log|J_{L, m}(e^{\frac{4\pi i}{2m+1}})|=Vol(M),$$
where $r=2m+1$ runs over all odd integers. 
\end{theorem} 

The asymptotic behavior of the values of $J_{L,m}(t)$ at $t=e^{\frac{2\pi i}{m+\frac{1}{2}}}$ is not predicted  either by the original  volume conjecture\,\cite{ kashaev:vol-conj,murakami:vol-conj} or by its generalizations \,\cite{gukov:vc, Murakami}. Theorem \ref{conjecturever} seems to suggest that these values  grow exponentially in $m$ with growth rate the hyperbolic volume. This is somewhat surprising because as noted in  \cite{gale}, and also in Corollary \ref{polynomial},  that for any positive integer $l,$ $J_{L,m}(e^{\frac{2\pi i}{m+l}})$ grows only polynomially in $m.$ We ask the following.

\begin{question}\label{dominanttermconj} Is it true that for any hyperbolic link $L$ in $S^3$, we have
$$\lim_{m\rightarrow +\infty}\frac{2\pi}{m}\log|J_{L,m}(e^{\frac{2\pi i}{m+\frac{1}{2}}})|=Vol(S^3\setminus L)?$$
\end{question}


\subsection{Knots with zero simplicial volume} Recall that the simplicial volume (or Gromov norm) $||L||$ of a link $L$ is the sum  of the volumes of the hyperbolic pieces in the JSJ-decomposition of the link complement, divided by the volume of the regular ideal hyperbolic tetrahedron. In particular,  if the geometric decomposition has no hyperbolic pieces, then $||L||=0$ \cite{Soma, thurston:notes}.
 As a natural generalization of Conjecture \ref{TVvolumeconj}, one can conjecture that for every link $L$ the asymptotics of $TV_r(S^3\setminus L, e^{\frac{2\pi i}{r}})$ determines $||L||.$ See Conjecture \ref{simplicial}.

 Using Theorem \ref{TV=CJP} and the positivity of the Turaev-Viro invariants (Corollary \ref{positivity}), we have a proof of  Conjecture \ref{simplicial} for the knots with zero simplicial volume.  
 \begin{theorem}\label{zero} Let $K\subset S^3$ be a knot with simplicial volume  zero. Then 

$$\lim_{r\to \infty} \frac {2\pi} {r} {\log TV_r(S^3\setminus K, e^{\frac{2\pi i}{r}})} =||K||=0,$$
where $r$ runs over all odd integers.
 \end{theorem}
 
 We also observe that, unlike the original volume conjecture that is not true for split links\,\cite[Remark 5.3]{murakami:vol-conj}, Conjecture \ref{simplicial} is compatible with  split unions of links, and under some assumptions is also compatible  with  connected summations. 
 
 Since this article was first written there has been some further progress in the study of relations of  the Turaev-Viro invariants and geometric decompositions of 3-manifolds:  
 By work of Ohtsuki \cite{ Ohtsuki}
 Conjecture \ref{TVvolumeconj} is true 
  for closed hyperbolic 3-manifolds obtained by integral surgeries along the figure-eight knot. In \cite{DKY}, the authors of this paper verify Conjecture \ref{TVvolumeconj} for infinite families of 
  cusped hyperbolic 3-manifolds. In \cite{DK},  Detcherry and Kalfagianni establish a relation between Turaev-Viro invariants and simplicial volume of 3-manifolds with empty or toroidal boundary, and proved generalizations of Theorem \ref{zero}. In \cite{D}, Detcherry proves that Conjecture  \ref{simplicial}  is stable under
certain link cabling operations.


\subsection{Organization} The paper is organized as follows. 
In Subsection \ref{RT}, we review the Reshetikhin-Turaev invariants\,\cite{ReTu}  following the skein theoretical approach by Blanchet, Habegger, Masbaum and Vogel\,\cite{BHMV, BHMV1, BHMV2}. In Subsection \ref{TVi}, we recall the definition of the Turaev-Viro invariants, and consider an $SO(3)$-version of them that facilitates our extension of the main theorem of \cite{BePe} in the setting needed in this paper  (Theorem \ref{RTdoubleTV}). The relationship between the two versions of the Turaev-Viro invariants is given in Theorem \ref{TVTV'} whose proof is included in  the Appendix. We prove Theorem \ref{TV=CJP} in Section \ref{main}, and prove Theorem \ref{conjecturever} and Theorem \ref{zero} respectively in Sections \ref{Applications} and \ref {general}.

\subsection{Acknowledgement} Part of this work was done while the authors were attending the conferences  ``Advances in Quantum and Low-Dimensional Topology 2016" at the University of Iowa, and ``Knots in Hellas 2016" at  the International Olympic Academy in Greece. We would like to thank the organizers of these conferences for support, hospitality, and for providing excellent working conditions.

We are also grateful to Francis Bonahon, Charles Frohman, Stavros Garoufalidis and Roland van der Veen for discussions and suggestions. 

\section{Preliminaries} \label{prelims}


\subsection{Reshetikhin-Turaev invariants and TQFTs}\label{RT} 
In this subsection we review the definition and basic properties of the Reshetikhin-Turaev invariants.
Our exposition follows the skein theoretical approach of Blanchet, Habegger, Masbaum and Vogel\,\cite{BHMV, BHMV1, BHMV2}.

A framed link in an oriented $3$-manifold $M$ is a smooth embedding $L$ 
 of a disjoint union of finitely many thickened circles $S^1\times [0,\epsilon],$ for some $\epsilon >0,$ into $M.$ Let $\Z[A,A^{-1}]$ be the ring of Laurent polynomials in the indeterminate $A.$ Then following \cite{Pr, Tu},  the \emph{Kauffman bracket skein module} $K_A(M)$ of $M$ is defined as the quotient of the free $\Z[A,A^{-1}]$-module generated by the isotopy classes of framed links in $M$ by the following two relations:
\begin{enumerate}[(1)]
\item  \emph{Kauffman Bracket Skein Relation:} \ $\cp{\includegraphics[width=1cm]{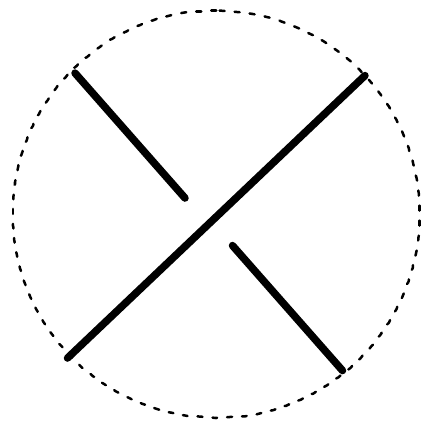}}\ =\ A\ \cp{\includegraphics[width=1cm]{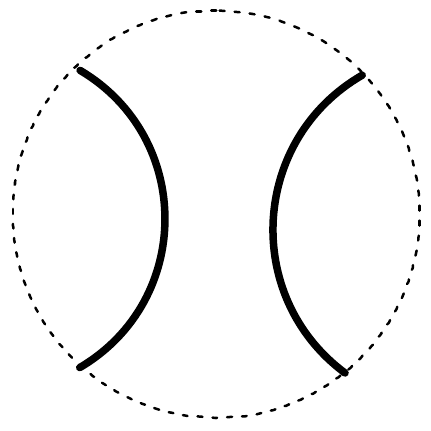}}\  +\ A^{-1}\ \cp{\includegraphics[width=1cm]{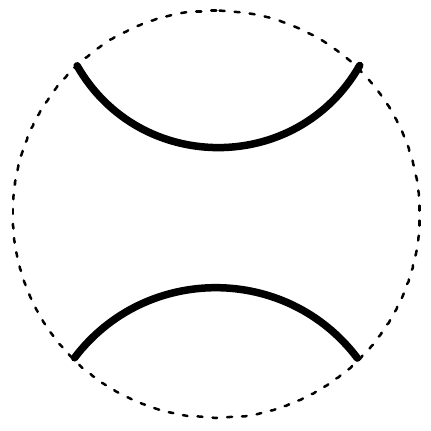}}.$ 

\item \emph{Framing Relation:} \ $L \cup \cp{\includegraphics[width=0.8cm]{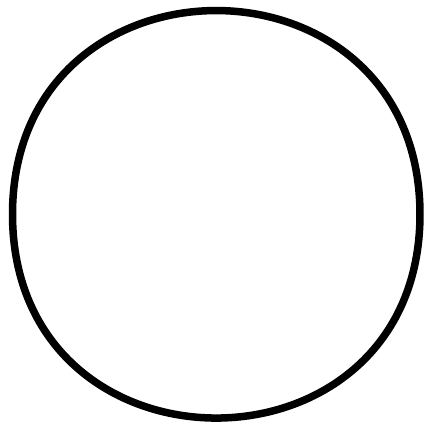}}=(-A^2-A^{-2})\ L.$ 
\end{enumerate}

There is a canonical isomorphism 
$$\langle\ \rangle: K_A(S^3) \rightarrow \mathbb Z[A,A^{-1}]$$
between the Kauffman bracket skein module of $S^3$ and $\mathbb Z[A,A^{-1}]$ viewed a module over itself. The Laurent polynomial $\langle L \rangle\in \mathbb Z[A,A^{-1}]$ determined by a framed link $L\subset S^3$ is called the \emph{Kauffman bracket} of $L.$

The Kauffman bracket skein module $K_A(T)$ of the solid torus $T=D^2\times S^1$ is canonically isomorphic to the module $\Z[A,A^{-1}][z].$ Here we consider $D^2$ as the unit disk in the complex plane, and call the framed link $[0,\epsilon]\times S^1 \subset D^2\times S^1,$ for some $\epsilon >0,$ the core of $T.$ Then the isomorphism above is given by sending $i$ parallel copies of the core of $T$ to $z^i.$ A framed link $L$ in $S^3$ of $n$ components defines an $\Z[A,A^{-1}]$-multilinear map
 $$\langle\ \ \ , \ldots ,\ \ \  \rangle_L : K_A(T)^{\otimes n} \rightarrow \Z[A,A^{-1}],$$
 called the \emph{Kauffman multi-bracket}, as follows. For monomials $z^{i_k}\in \Z[A,A^{-1}][z]\cong K_A(T),$ $k = 1, \dots, n,$ let $L(z^{i_1}, \dots, z^{i_n})$ be the framed link in $S^3$ obtained by cabling the $k$-th component of $L$ by $i_k$  parallel copies of the core. Then define $$\langle z^{i_1}, \dots, z^{i_n} \rangle_L \doteq \langle L(z^{i_1}, \dots, z^{i_n}) \rangle,$$ and extend $\Z[A,A^{-1}]$-multilinearly on the whole $K_A(T).$  For  the unknot $U$ and  any polynomial $P(z)\in \Z[A,A^{-1}][z],$ we simply denote  the bracket $\langle P(z) \rangle_U$ by $\langle P(z) \rangle.$
 
 The $i$-th Chebyshev polynomial $e_i \in \Z[A,A^{-1}][z]$ is defined by the recurrence relations
 $e_0=1,$ $e_1=z,$ and  $ze_j=e_{j+1}+e_{j-1},$ and satisfies 
 $$\langle e_i \rangle = (-1)^i[i+1].$$ 
The colored Jones polynomials of an oriented  knot $K$ in $S^3$ are defined using $e_i$ as follows. Let $D$ be a diagram of $K$ with writhe number $w(D),$ equipped with the blackboard framing. Then the $(i+1)$-st colored Jones polynomial of $K$ is 
\begin{equation*}
J_{K,i+1}(t)=\big((-1)^i A^{i^2+2i}\big)^{w(D)} \langle e_i \rangle_D.
\end{equation*}

The colored Jones polynomials for an oriented link $L$ in $S^3$ is defined similarly. Let $D$ be a diagram of $L$ with writhe number $w(D)$ and equipped with the  blackboard framing. For a multi-integer $\mathbf i = (i_1, \dots, i_n),$ let $\mathbf i+\mathbf 1 = (i_1+1,\dots, i_n+1).$ Then the $(\mathbf i+\mathbf 1)$-st colored Jones polynomial of $L$ is defined by

$$J_{L, \mathbf i+\mathbf 1}(t) = \big((-1)^{\underset{k=1}{\overset{n}{\sum}}i_k }A^{s({\mathbf i})}\big)^{w(D)}\langle e_{i_1}, \dots, e_{i_n} \rangle _D,$$
where  $s({\mathbf i})={\underset{k=1}{\overset{n}{\sum}}(i_k^2+i_k)}.$

We note that a change of orientation on some or all the components  of $L$ changes the writhe number of $D,$ and changes $J_{L,\mathbf i}(t)$ only by a power of $A.$ Therefore, for an unoriented link $L$ and a complex number $A$ with $|A|=1,$ the modulus of the value of $J_{L,\mathbf i}(t)$ at $t=A^4$ is well defined, and
\begin{equation}\label{relation}
|J_{L, \mathbf i}(t)| =|\langle e_{i_1-1}, \dots, e_{i_n-1} \rangle _D|.
\end{equation}

 If $M$ is a closed oriented $3$-manifold obtained by doing surgery along a framed link $L$ in $S^3,$ then the specialization of the Kauffman multi-bracket at roots of unity yields invariants of $3$-manifolds. From now on, let $A$ be either a primitive $4r$-th root of unity for an integer $r\geqslant 3$ or a primitive $2r$-th root of unity for an odd integer $r\geqslant 3.$ To define the Reshetikhin-Turaev invariants, we need to recall some special elements of $K_A(T)\cong\Z[A,A^{-1}][z],$ called the \emph{Kirby coloring}, defined by
$$\omega_{r}=\underset{i=0}{\overset{r-2}{\sum} }\langle e_i \rangle e_i$$
for any integer $r,$ and 
$$\omega_{r}'=\underset{i=0}{\overset{m-1}{\sum}}\langle e_{2i} \rangle e_{2i}$$
for any odd integer $r=2m+1.$ We also for any $r$ introduce
$$\kappa_{r}=\eta_r \langle \omega_r \rangle_{U_+}, $$
and for any odd $r$ introduce
$$\kappa'_{r}=\eta_r' \langle \omega_r' \rangle_{U_+}, $$
where $U_{+}$ is the unknot with framing $1.$

\begin{definition}\label{RTinvariants} Let $M$ be a closed oriented $3$-manifold obtained from $S^3$ by doing surgery along a framed link $L$ with number of components $n(L)$ and signature $\sigma(L).$ 
\begin{enumerate}[(1)]
\item The Reshetikhin-Turaev invariants  of $M$ are  defined by $$\langle M\rangle_{r}=\eta_{r}^{1+n(L)}\ \kappa_{r}^{-\sigma(L)}\ \langle \omega_r, \dots, \omega_r \rangle _{L}$$ for any integer $r\geqslant 3,$ and by
$$\langle M\rangle_{r}'=(\eta_{r}')^{1+n(L)}\ (\kappa_{r}')^{-\sigma(L)}\ \langle \omega_{r}', \dots, \omega_{r}'\rangle _{L}$$ for any odd integer $r\geqslant 3.$ 
\item Let $L'$ be a framed link in $M.$ Then, the Reshetikhin-Turaev invariants of the pair $(M,L')$ are defined by $$\langle M,L' \rangle_{r}=\eta_{r}^{1+n(L)}\ \kappa_{r}^{-\sigma(L)}\ \langle \omega_r, \dots, \omega_r, 1 \rangle _{L\cup L'}$$ for any integer $r\geqslant 3,$ and by
$$\langle M, L' \rangle_{r}'=(\eta_{r}')^{1+n(L)}\ (\kappa_{r}')^{-\sigma(L)}\ \langle \omega_{r}', \dots, \omega_{r}', 1 \rangle _{L\cup L'}$$ for any odd integer $r\geqslant 3.$
\end{enumerate}
\end{definition}

\begin{remark}\label{sphere}
\begin{enumerate}[(1)]
\item The invariants  $\langle M\rangle_{r}$ and $\langle M \rangle _r'$ are called  the $SU(2)$ and $SO(3)$ Reshetikhin-Tureav invariants of $M,$ respectively.
\item For any element $S$ in $K_A(M)$ represented by a $\Z[A,A^{-1}]$-linear combinations of framed links in $M,$ one can define $\langle M, S \rangle_r$ and $\langle M, S \rangle_r'$ by $\Z[A,A^{-1}]$-linear extensions. 
\item
Since $S^3$ is obtained by doing surgery along the empty link, we have $\langle S^3 \rangle_{r}=\eta_{r}$ and $\langle S^3 \rangle_r' =\eta_r'.$ Moreover, 
 for any link $L \subset S^3$ we have
$$\langle S^3 ,L \rangle_{r} =\eta_{r}\  \langle L \rangle, \ \ 
{\rm and } \ \ 
\langle S^3, L  \rangle_{r}'= \eta_{r}'\  \langle L \rangle.  $$
\end{enumerate}
\end{remark}
\bigskip

In \cite{BHMV2}, Blanchet, Habegger, Masbaum and Vogel gave a construction of  the topological quantum field theories underlying the $SU(2)$ and $SO(3)$ versions of the  Reshetikhin-Turaev invariants.
Below we will summarize the basic properties of the corresponding  topological quantum field functors denoted by
 $Z_r$ and $Z_{r'}$, respectively.  Note that for a  closed oriented 3-manifold $M$ we will use $-M$ to denote the manifolds with the orientation reversed.

\begin{theorem} \cite[Theorem 1.4]{BHMV2} \begin{enumerate}[(1)]\label{TQFT}

\item For a closed oriented surface $\Sigma$  and  any integer $r\geqslant 3,$ there exists a finite dimensional $\mathbb{C}$-vector space $Z_r(\Sigma)$ satisfying $$Z_r(\Sigma_1 \coprod \Sigma_2) \cong Z_r(\Sigma_1)\otimes Z_r(\Sigma_2),$$
and, similarly,  for each odd integer $r\geqslant 3,$ 
there exists a finite dimensional $\mathbb{C}$-vector space $Z_r'(\Sigma)$ satisfying $$Z_r'(\Sigma_1 \coprod \Sigma_2) \cong Z_r'(\Sigma_1)\otimes Z_r'(\Sigma_2).$$

\item If $H$ is a handlebody with $\partial H =\Sigma,$ then $Z_r(\Sigma)$ and $Z_r'(\Sigma)$ are  quotients of the Kauffman bracket skein module $K_A(H).$

\item \label{relation2} Every compact oriented $3$-manifold $M$ with $\partial M=\Sigma$ and a framed link $L$ in $M$ defines for any integer $r$ a vector $Z_r(M,L)$  in $Z_r(\Sigma),$  and for any odd integer $r$ a vector $Z_r'(M,L)$  in $Z_r'(\Sigma).$ 

\item For  any integer $r,$ there is a sesquilinear pairing $\langle\ ,\ \rangle$ on $Z_r(\Sigma)$ with the following property: Given  oriented 3-manifolds $M_1$ and $M_2$  with boundary  $\Sigma=\partial M_1=\partial M_2$, and framed links $L_1\subset M_1$ and $L_2\subset M_2,$   we have

 \begin{equation*}
 \langle M, L \rangle_{r}=\langle Z_{r}(M_1, L_1),Z_{r}(M_2, L_2)\rangle,
 \end{equation*}

where $M=M_1\bigcup_{\Sigma} (-M_2)$ is the closed 3-manifold obtained by gluing $M_1$ and $-M_2$  along  $\Sigma$ and $L= L_1\coprod L_2$. 
Similarly, for any odd integer $r$, there is a sesquilinear pairing $\langle\ ,\ \rangle$ 
on  $Z_r'(\Sigma)$,  such tor any  $M$ and $L$ as above,
  \begin{equation*}
 \langle M, L \rangle_{r}'=\langle Z_{r}'(M_1, L_1),Z_{r}'(M_2, L_2)\rangle.
 \end{equation*}
\end{enumerate}
\end{theorem}

For the purpose of this paper, we will only need to understand the TQFT vector spaces of the torus $Z_r(T^2)$ and $Z_r'(T^2).$ These vector spaces are quotients of $K_A(T)\cong\Z[A,A^{-1}][z],$ hence the Chebyshev polynomials $\{e_i\}$ define vectors in $Z_r(T^2)$ and $Z_r'(T^2).$  We have the following

\begin{theorem}\cite[Corollary 4.10, Remark 4.12]{BHMV2} \label{TQFTbasis}
\begin{enumerate}[(1)]
\item For any integer $r\geqslant 3,$ the vectors $\{e_0, \dots, e_{r-2}\}$ form a Hermitian basis of $Z_{r}(T^2).$ 
\item For any odd integer $r=2m+1,$ the vectors $\{e_0, \dots, e_{m-1}\}$ form a Hermitian basis of $Z_{r}'(T^2).$
\item In $Z_r'(T^2),$ we have for any $i$ with  $0\leqslant i \leqslant m-1$ that 
 \begin{equation}\label{duality}
 e_{m+i}=e_{m-1-i}.
 \end{equation}  Therefore, the vectors $\{e_{2i}\}_{i=0,\ldots ,m-1}$ also form a Hermitian basis of $Z_r'(T^2).$
 \end{enumerate}
\end{theorem}


\subsection{Turaev-Viro invariants}\label{TVi}

In this subsection, we recall the definition and basic properties of the Turaev-Viro invariants\,\cite{TuraevViro, KauLins}.
The approach of \cite{TuraevViro} relies on quantum 6$j$-symbols  while the definition of  Kauffman and Lins \cite{KauLins} uses invariants of spin networks.
The two definitions were shown to be equivalent in \cite{KLTV}.  The formalism of \cite{ KauLins}  turns out to be more convenient to work with when using skein theoretic techniques to relate the Turaev-Viro invariants to the Reshetikhin-Turaev invariants.

For an integer $r\geqslant 3,$ let $I_r=\{0,1,\dots, r-2\}$ be the set of non-negative integers less than or equal to $r-2.$  Let $q$ be a $2r$-th root of unity such that $q^2$ is a primitive $r$-th root. For example, $q=A^2,$ where $A$ is either a primitive $4r$-th root or for odd $r$ a primitive $2r$-th root, satisfies the condition.  For $i\in I_r,$ define 
$$\cp{\includegraphics[width=1cm]{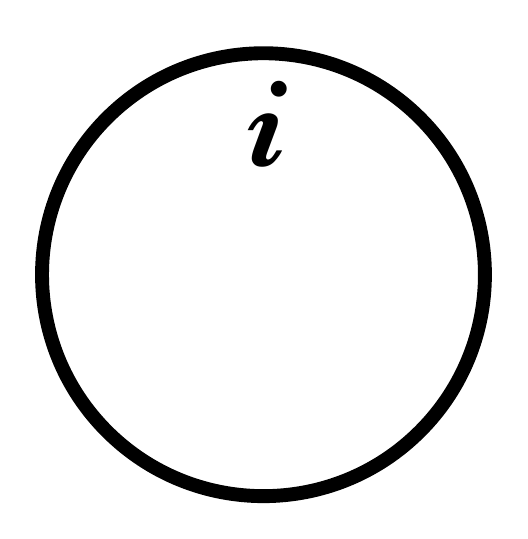}}=(-1)^{i}[i+1].$$
A triple $(i,j,k)$ of elements of $I_r$ is called \emph{admissible} if (1) $i+j\geqslant k,$ $j+k\geqslant i$ and $k+i\geqslant j,$ (2) $i+j+k$ is an even,
and (3) $i+j+k\leqslant 2(r-2).$ For an admissible triple $(i,j,k),$ define

\begin{equation*}
\cp{\includegraphics[width=1.25cm]{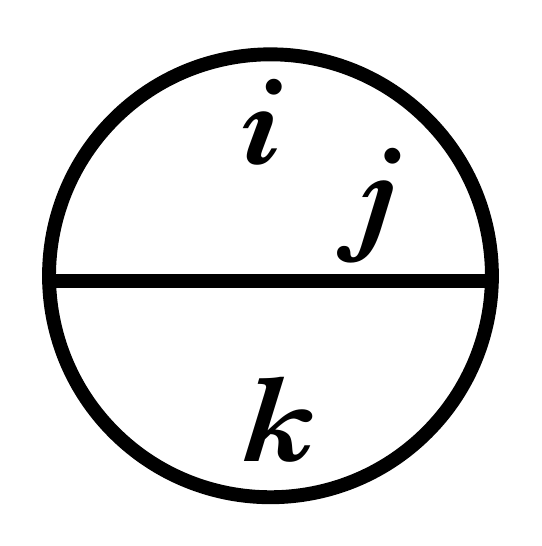}} = (-1)^{-\frac{i+j+k}{2}}\frac{[\frac{i+j-k}{2}]![\frac{j+k-i}{2}]![\frac{k+i-j}{2}]![\frac{i+j+k}{2}+1]!}{[i]![j]![k]!}.
\end{equation*}

A $6$-tuple $(i,j,k,l,m,n)$ of elements of $I_r$ is called \emph{admissible} if the triples $(i,j,k),$ $(j,l,n),$ $(i,m,n)$ and $(k,l,m)$ are admissible. For an admissible $6$-tuple $(i,j,k,l,m,n),$
define 
\begin{equation*}
\begin{split}
\cp{\includegraphics[width=1.5cm]{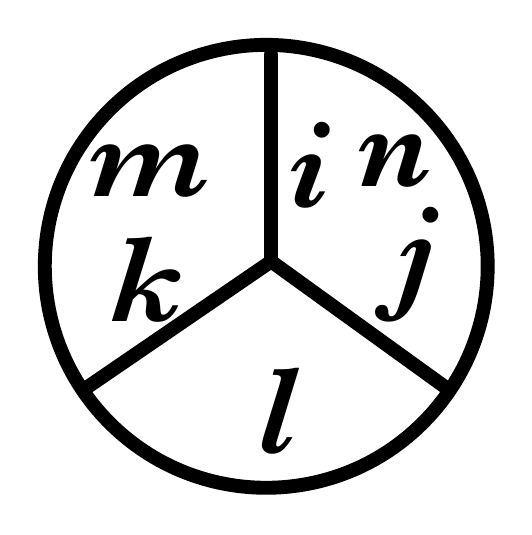}}=&\ \frac{\prod_{a=1}^4\prod_{b=1}^3[Q_b-T_a]!}{[i]![j]![k]![l]![m]![n]!}
\sum_{z=\max \{T_1, T_2, T_3, T_4\}}^{\min\{ Q_1,Q_2,Q_3\}}\frac{(-1)^z[z+1]!}{\prod_{a=1}^4[z-T_a]!\prod_{b=1}^3[Q_b-z]!},
\end{split}
\end{equation*}
where

 $$T_1=\frac{i+j+k}{2}, \ \  T_2=\frac{i+m+n}{2}, \ T_3=\frac{j+l+n}{2}, \  \ T_4=\frac{k+l+m}{2},$$ 
$$Q_1=\frac{i+j+l+m}{2},\ \  Q_2=\frac{i+k+l+n}{2}, \ \ Q_3=\frac{j+k+m+n}{2}.$$

\begin{remark}\label{notation}  The symbols  $\cp{\includegraphics[width=1cm]{circle}}, \cp{\includegraphics[width=1cm]{theta}}$ and $\cp{\includegraphics[width=1cm]{Benz}}$, used above, 
are examples of spin networks: trivalent ribbon graphs with ends colored by integers.  The expressions on the right hand sides of above equations give the Kauffman bracket invariant of the
 corresponding networks. See \cite[Chapter 9]{KauLins}. In the language of \cite{KauLins}, the
second and third spin networks above are the  trihedral and tetrahedral networks, 
denoted by $\theta(i,j,k)$ and $\tau(i, j,k)$ therein, 
and the corresponding invariants are  the trihedral and  tetrahedral coefficients, respectively.
\end{remark}

\begin{figure}[htbp]\centering
\includegraphics[width=4.5cm]{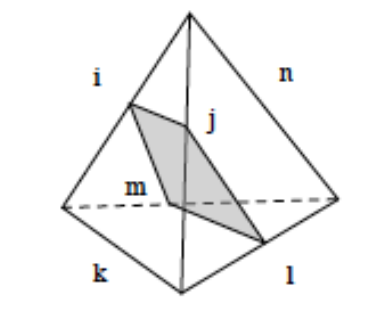}\\
\caption{ The quantities $T_1, \dots, T_4$ correspond to faces and   $Q_1, Q_2, Q_3$ 
correspond to quadrilaterals.
}\label{tetra}
\end{figure}

\begin{definition}A \emph{coloring} of a Euclidean tetrahedron $\Delta$ is an assignment of elements of $I_r$ to the edges of $\Delta,$ and is \emph{admissible} if the triple of elements of $I_r$ assigned to the three edges of each face of $\Delta$ is admissible. See Figure~\ref{tetra} for a geometric interpretation of tetrahedral coefficients. \end{definition}

Let $\mathcal T$ be a triangulation of $M.$ If $M$ has non-empty boundary, then we let $\mathcal T$ be an ideal triangulation of $M,$ i.e., a gluing of finitely many truncated Euclidean tetrahedra by affine homeomorphisms between pairs of faces. In this way, there are no vertices, and instead, the triangles coming from truncations form a triangulation of the boundary of $M.$ By edges of an ideal triangulation, we only mean the ones coming from the edges of the tetrahedra, not the ones from the truncations.  A \emph{coloring at level $r$} of the triangulated $3$-manifold $(M,\mathcal T)$ is an assignment of elements of $I_r$ to the edges of $\mathcal T,$ and is \emph{admissible} if the $6$-tuple assigned to the edges of each tetrahedron of $\mathcal T$ is admissible. Let $c$ be an admissible coloring of $(M,\mathcal T)$ at level $r.$ For each edge $e$ of $\mathcal T,$ let 
$$|e|_c=\cp{\includegraphics[width=1.5cm]{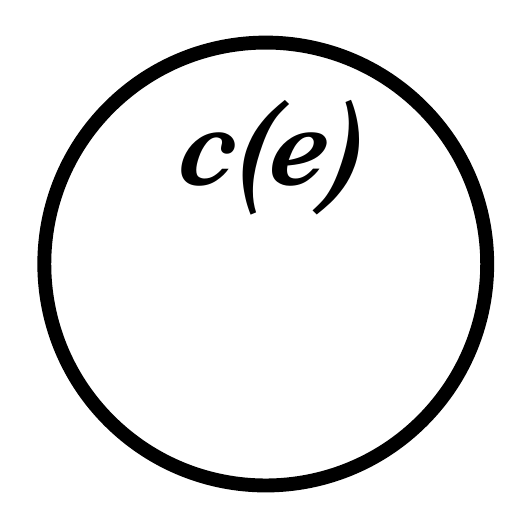}}.$$ For each face $f$ of $\mathcal T$ with edges $e_1, e_2$ and $e_3,$ let 

\begin{equation*} |f|_c=\cp{\includegraphics[width=1.75cm]{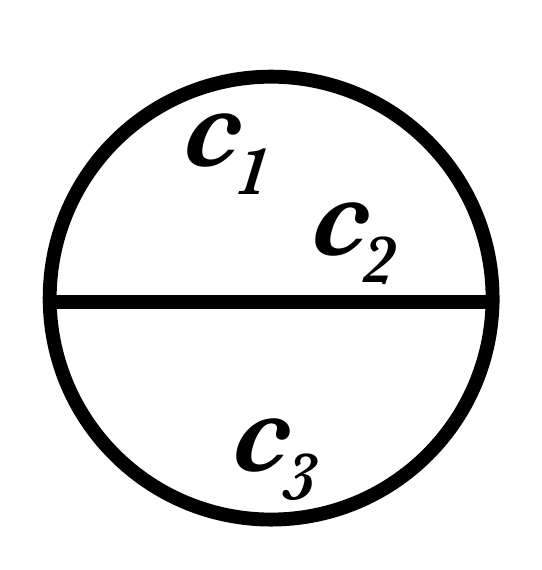}},
\end{equation*}
where $c_i = c(e_i).$

For each tetrahedra $\Delta$ in $\mathcal T$ with vertices $v_1,\dots,v_4,$ denote by $e_{ij}$ the edge of $\Delta$ connecting the vertices $v_i$ and $v_j,$ $\{i,j\}\subset\{1,\dots,4\},$ and let
\begin{equation*}
|\Delta|_c=\cp{\includegraphics[width=2cm]{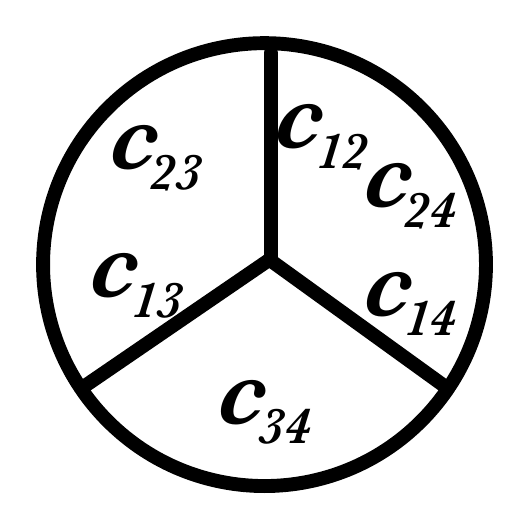}},
\end{equation*}
where $c_{ij}=c(e_{ij}).$

\begin{definition}\label{TV} Let $A_r$ be the set of admissible colorings of $(M,\mathcal T)$ at level $r,$ and let $V,$ $E$ $F$ and  $T$ respectively be the sets of  (interior)  vertices, edges, faces and tetrahedra in $\mathcal T.$ Then the $r$-th Turaev-Viro invariant is defined by
 $$TV_r(M)= \eta_r^{2|V|}\sum_{c\in A_r}\frac{\prod_{e\in E}|e|_c\prod_{\Delta\in T}|\Delta|_c}{\prod_{f\in E}|f|_c}.$$
\end{definition}

For an odd integer $r\geqslant 3,$ one can also consider an $SO(3)$-version of the Turaev-Viro invariants $TV_r'(M)$  of  $M,$ which will relate  to the $SU(2)$ invariants $TV_r(M)$,  and to the Reshetikhin-Turaev invariants $\langle D(M) \rangle'_r$ of the double of $M$ (Theorems \ref{TVTV'}, \ref{RTdoubleTV}). The invariant $TV_r'(M)$ is defined as follows. Let $I'_r=\{0,2,\dots, r-5, r-3\}$ be the set of non-negative even integers less than or equal to $r-2.$  An \emph{$SO(3)$-coloring} of a Euclidean tetrahedron $\Delta$ is an assignment of elements of $I'_r$ to the edges of $\Delta,$ and is \emph{admissible} if the triple assigned to the three edges of each face of $\Delta$ is admissible. Let $\mathcal T$ be a triangulation of $M.$ An \emph{$SO(3)$-coloring at level $r$} of the triangulated $3$-manifold $(M,\mathcal T)$ is an assignment of elements of $I'_r$ to the edges of $\mathcal T,$ and is \emph{admissible} if the $6$-tuple assigned to the edges of each tetrahedron of $\mathcal T$ is admissible. 

\begin{definition}\label{TV'} Let $A'_r$ be the set of $SO(3)$-admissible colorings of $(M,\mathcal T)$ at level $r.$  Define
 $$TV_r'(M)=(\eta_r')^{2|V|}\sum_{c\in A'_r}\frac{\prod_{e\in E}|e|_c\prod_{\Delta\in T}|\Delta|_c}{{\prod_{f\in E}|f|_c}}.$$
\end{definition}

The relationship between $TV_r(M)$ and $TV'_r(M)$ is given by the following theorem.

\begin{theorem}\label{TVTV'} Let $M$ be a $3$-manifold and let $b_0(M)$ and $b_2(M)$ respectively be its zeroth and second $\mathbb Z_2$-Betti number. 
\begin{enumerate}[(1)]
\item For any odd integer $r\geqslant 3,$ 
$$TV_r(M)=TV_3(M)\cdot TV'_r(M).$$
\item (Turaev-Viro\,\cite{TuraevViro}). If $\partial M= \emptyset$ and $A=e^{\frac{\pi i}{3}},$ then $$TV_3(M)=2^{b_2(M)-b_0(M)}.$$
\item If $M$ is connected, $\partial M\neq\emptyset$ and $A=e^{\frac{\pi i}{3}},$ then $$TV_3(M)=2^{b_2(M)}.$$
\end{enumerate}
In particular, $TV_3(M)$ is nonzero.
\end{theorem}

We postpone the proof of Theorem \ref{TVTV'} to Appendix \ref{A} to avoid unnecessary distractions.


\section{The colored Jones sum formula for Turaev-Viro invariants} \label{main}
In this Section, following the argument of \cite{BePe}, we establish a relationship between the $SO(3)$ Turaev-Viro invariants of a 3-manifold with boundary and the $SO(3)$ Reshetikhin-Turaev invariants of its double. See Theorem \ref{RTdoubleTV}. Then, we use  Theorem \ref{RTdoubleTV} and results established  in \cite{BHMV, BHMV2},  to prove Theorem \ref{TV=CJP}.

\subsection{Relationship between invariants}\label{BP}

The relationship between  Turaev-Viro and Witten-Reshetikhin-Turaev invariants 
was studied by Turaev-Walker\,\cite{Turaevbook} and Roberts\,\cite{Roberts} for closed 3-manifolds, and by Benedetti  and Petronio\,\cite{BePe} for 3-manifolds with boundary.
For an oriented 3-manifold $M$ with boundary, let $-M$  denote $M$ with the orientation reversed, and  let $D(M)$
denote the double of $M,$ i.e.,
$$D(M)=M\underset{\partial M}{\bigcup}(-M).$$
We will need the following theorem of Benedetti and Petronio\,\cite{BePe}. In fact \cite{BePe} only treats the case of $A=e^{\frac{\pi i}{2r}},$ but, as we will explain below, the proof for other cases is similar.

\begin{theorem}[]\label{RTdoubleTV}
Let $M$ be a $3$-manifold with boundary.
Then,  
$$TV_r(M)=\eta_{r}^{-\chi(M)}\langle D(M) \rangle_{r}$$
for any integer $r,$ and
$$TV_r'(M)=(\eta'_{r})^{-\chi ( M)}\langle D(M) \rangle_{r}'$$
for any odd $r,$ where $\chi(M)$ is the Euler characteristic of $M.$
\end{theorem}

We refer to \cite{BePe} and \cite{Roberts} for the $SU(2)$ ($r$ being any integer) case, and for the reader's convenience include a sketch of the proof here for the $SO(3)$ ($r$ being odd) case. The main difference for the $SO(3)$ case comes from to the following lemma due to Lickorish.

\begin{lemma}[{\cite[Lemma 6]{Lickorish}}]\label{Lickorish}
Let $r\geqslant 3$ be an odd integer and let $A$ be a primitive $2r$-th root of unity.  Then 
\begin{equation*}
\cp{\includegraphics[width=1.8cm]{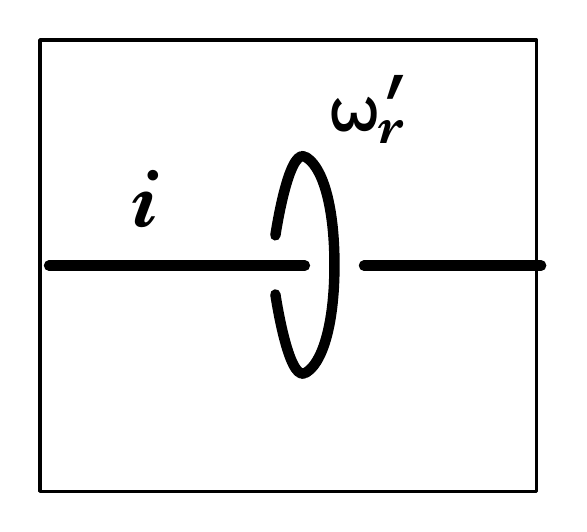}}=\left\{
\begin{array}{cl}
\cp{\includegraphics[width=1.8cm]{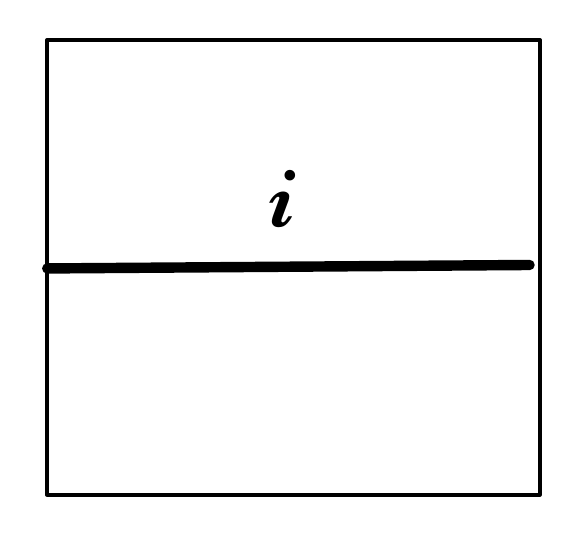}} &\quad \text{if }  i=0, r-2\\
&\\
0 &\quad \text{if } i\neq 0, r-2
\end{array}\right.
\end{equation*}
I.e., the element of the $i$-th Temperley-Lieb algebra obtained by circling the $i$-th Jones-Wenzl idempotent $f_i$ by the Kirby coloring $\omega_r'$ equals $f_i$ when $i=0$ or $r-2,$ and equals $0$ otherwise. 
\end{lemma}
As a consequence, the usual fusion rule\,\cite{lickorish:book} should be modified to the following.

\begin{lemma}[Fusion Rule]\label{FR} Let $r\geqslant 3$ be an odd integer. Then for a triple $(i,j,k)$ of elements of $I_r',$ 

\begin{equation*}
\cp{\includegraphics[width=2.2cm]{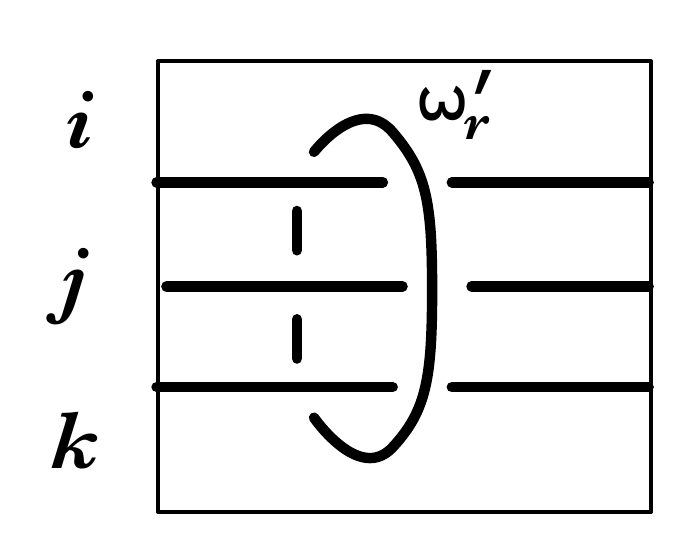}}=\left\{
\begin{array}{cl}
\frac{(\eta_r')^{-2}}{\cp{\includegraphics[width=1cm]{theta}}}\cp{\includegraphics[width=2.5cm]{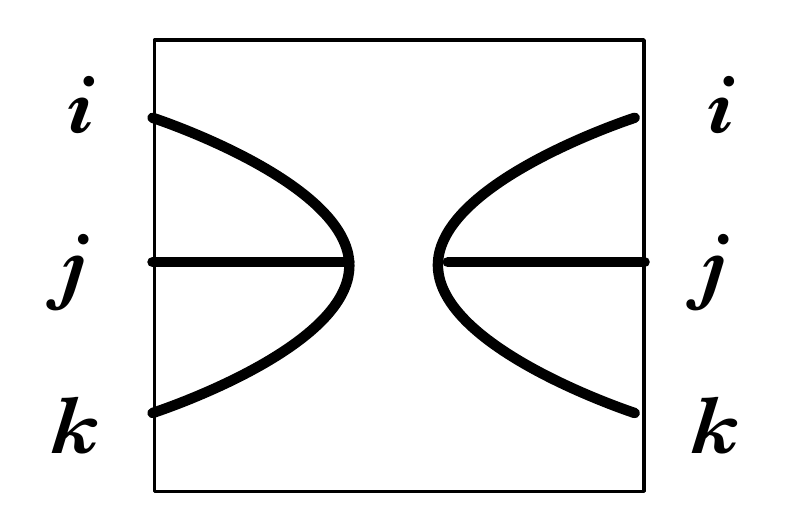}} & \quad\text{if } (i,j,k) \text{ is $r$-admissible,}\\
&\\
0 & \quad\text{if } (i,j,k) \text{ is not $r$-admissible.}
\end{array}\right.
\end{equation*}
\end{lemma}

Here the integers $i,j$ and $k$ being even is crucial, since it rules out the possibility that $i+j+k=r-2,$ which by Lemma \ref{Lickorish} could create additional complications. This is the reason that we prefer to work with  the invariant $TV_r'(M)$ instead of $TV_r(M).$  Note that the factor  $\cp{\includegraphics[width=0.8cm]{theta}}$ in the formula above is also denoted by $\theta(i,j,k)$ in \cite{BePe} and \cite{Roberts}.

\begin{proof}[Sketch of proof of Theorem \ref{RTdoubleTV}]

Following \cite{BePe}, we extend the ``chain-mail" invariant of Roberts\,\cite{Roberts} to $M$ with non-empty boundary using a handle decomposition without $3$-handles. For such a handle decomposition, let $d_0, d_1$ and $d_2$ respectively be the number of $0$-, $1$- and $2$-handles.  Let $\epsilon_i$ be the attaching curves of the $2$-handles and let $\delta_j$ be the meridians of the $1$-handles. Thicken the curves to bands parallel to the surface of the $1$-skeleton $H$ and push the $\epsilon$-bands slightly into $H.$   Embed $H$ arbitrarily into $S^3$ and color each of the image of the $\epsilon$- and $\delta$-bands by $\eta_r'\omega_r'$ to get an element in $S_M$ in $K_A(S^3).$ Then the chain-mail invariant of $M$ is defined by
$$CM_r(M)=(\eta_r')^{d_0}\langle S_M \rangle,$$  where, recall that, we use the notation $\langle \ \rangle$ for the Kauffman bracket.  It is proved in \cite{BePe, Roberts} that $CM_r(M)$ is independent of the choice of the handle decomposition and the embedding, hence defines an invariant of $M.$

To prove the result we will compare the expressions of the invariant  $CM_r(M)$  obtained by considering two different handle decompositions of $M.$
On the one hand, suppose that  the handle decomposition is obtained by the dual of an ideal triangulation $\mathcal T$ of $M,$ namely the $2$-handles come from a tubular neighborhood of the edges of $\mathcal T,$ the $1$-handles come from a tubular neighborhood of the faces of $\mathcal T$ and the $0$-handles come from the complement of the $1$- and $2$-handles. Since each face has  three edges,  each $\delta$-band encloses exactly three $\epsilon$-bands (see \cite[Figure 11]{Roberts}). By relation (\ref{duality}), every $\eta_r'\omega_r'$ on the $\epsilon$-band can be written as 
$$\eta_r'\omega_r'=\eta_r'\sum_ {i=0}^{\frac{r-1}{2}-1}\langle e_{i}\rangle e_{i}=\eta_r'\sum_ {i=0}^{\frac{r-1}{2}-1}\langle e_{2i}\rangle e_{2i}.$$
Next we apply Lemma \ref{FR} to each $\delta$-band. In this process the four $\delta$-bands corresponding to each 
tetrahedron of $\mathcal T$  give rise to a tetrahedral network
(see also \cite[Figure 12]{Roberts}).  Then by Remark \ref{notation} and equations preceding it, we may rewrite $CM_r(M)$ in terms of trihedral and tetrahedral coefficients
to obtain 

$$CM_r(M)=(\eta_r')^{d_0-d_1+d_2}\sum_{c\in A'_r}\frac{\prod_{e\in E}|e|_r^c\prod_{\Delta\in T}|\Delta|_r^c}{\prod_{f\in E}|f|_r^c}=(\eta_r')^{\chi(M)}TV_r'(M).$$

On the other hand, suppose that the handle decomposition is standard, namely $H$ is a standard handlebody in $S^3$ with exactly one $0$-handle. Then we claim that the $\epsilon$- and the $\delta$-bands give a surgery diagram $L$ of $D(M).$ The way to see it is as follows. Consider the $4$-manifold $W_1$ obtained by attaching $1$-handles along the $\delta$-bands (see Kirby\,\cite{Kirby}) and $2$-handles along the $\epsilon$-bands. Then $W_1$ is homeomorphic to $M\times I$ and $\partial W_1 =M\times\{0\}\cup \partial M\times I \cup(-M)\times\{1\}= D(M).$ Now if $W_2$ is the $4$-manifold obtained by attaching $2$-handles along all the $\epsilon$- and the $\delta$-bands, then $\partial W_2$ is the $3$-manifold represented by the framed link $L.$ Then due to the fact that $\partial W_1=\partial W_2$ and Definition \ref{RTinvariants}, we have $$CM_r(M)=\eta_r'\langle \eta_r'\omega_r',\dots, \eta'\omega' \rangle_L=(\eta_r')^{1+n(L)}\langle \omega_r',\dots, \omega_r' \rangle_L=\langle D(M)\rangle_r' (\kappa_r') ^{\sigma(L)}.$$ 
We are left to show that $\sigma(L)=0.$ It follows from the fact that the linking matrix of $L$ has the form 
\begin{equation*}
LK(L)=\begin{bmatrix}
0 & A \\
A^T&0 \end{bmatrix},
\end{equation*}
where the blocks come from grouping the $\epsilon$- and the $\delta$-bands together and $A_{ij}=LK(\epsilon_i, \delta_j).$ Then, for any eigenvector $v=(v_1,v_2)$ with eigenvalue $\lambda,$ the vector $v'=(-v_1,v_2)$ is an eigen-vector of eigenvalue $-\lambda.$ 
\end{proof}

\begin{remark}\label{double} Theorems \ref{RTdoubleTV} and \ref{TVTV'} together with the main result of \cite{Roberts}  imply that if
Conjecture \ref{TVvolumeconj} holds for $M$ with totally geodesic or toroidal boundary, then it holds for $D(M).$

\end{remark}


\subsection{Proof of Theorem \ref{TV=CJP}}
We are now ready to prove Theorem \ref{TV=CJP}.  For the convenience of the reader we restate the theorem.

\begin{named}{Theorem \ref{TV=CJP}}
Let $L$ be a link in $S^3$ with $n$ components. 
\begin{enumerate}[(1)]
\item
For an integer $r\geqslant 3$  and a primitive $4r$-th root of unity $A,$ we have
$$TV_r(S^3\setminus L,q)=\eta_{r}^{2}\underset{1 \leqslant \mathbf{i}\leqslant r-1}{\sum} 
|  J_{L, \mathbf{i}} (t)   |^2.$$ 
\item
For an odd integer  $r=2m+1\geqslant 3$ and a primitive $2r$-th root of unity $A,$ we have
$$TV_r(S^3\setminus L,q)=2^{n-1}(\eta_{r}')^{2}\underset{1 \leqslant \mathbf{i}\leqslant m}{\sum}
|  J_{L, \mathbf{i}} (t)|^2.$$ 

Here, in both cases we have $t=q^2=A^4.$
\end{enumerate}
\end{named}

\begin{proof}We first consider the case that  $r=2m+1$ is odd. For a framed link $L$ in $S^3$ with $n$ components, we let $M=S^3\setminus L.$
Since, by Theorem \ref{TVTV'}, we have
$TV_r(M)=2^{n-1} TV_r'(M)$, from now on we will work with $TV_r'(M).$

Since the Euler characteristic of $M$ is zero,
by  Theorem \ref{RTdoubleTV}, we obtain
\begin{equation}
TV_r'(M)=\langle D(M) \rangle_r'=\langle Z_r'(M),Z_r'(M) \rangle,
\label{applydouble}
\end{equation}
where $Z_r(M)$ is a vector in $Z_r(T^2)^{\otimes n}.$
Let $\{e_i\}_{i=0,\dots, m-1}$ be the basis of $Z_r'(T^2)$ described in Theorem \ref{TQFTbasis} (2). 
Then the vector space $Z_r(T^2)^{\otimes n}$ has a Hermitian basis given by $\{e_{\mathbf{i}}=e_{i_1}\otimes e_{i_2} \ldots e_{i_n}\}$ for all $\mathbf{i}=(i_1,i_2,\ldots,i_n)$ with $0\leqslant \mathbf i \leqslant m-1.$
 
We write $\langle e_{\mathbf{i}} \rangle_L$ for the multi-bracket $\langle e_{i_1},e_{i_2},\ldots,e_{i_n} \rangle_L.$ Then, by relation (\ref{relation}), to establish the desired formula in terms of the colored Jones polynomials, it is suffices to show that
 $$TV_r'(M)=(\eta_{r}')^{2}\underset{0 \leqslant \mathbf{i}\leqslant m-1}{\sum}|\langle e_{\mathbf{i}}\rangle_L|^2.$$ 
By writing $$Z_r'(M)=\underset{0 \leqslant \mathbf{i} \leqslant m-1}{\sum} \lambda_{\mathbf{i}}e_{\mathbf{i}}$$
and  using equation (\ref{applydouble}), we have that 
$$TV_r'(M)=\underset{0 \leqslant \mathbf{i} \leqslant m-1}{\sum}|\lambda_{\mathbf{i}}|^2.$$

The computation of the coefficients  $\lambda_{\mathbf{i}}$ of $Z_r(M)$   relies on the TQFT properties of the invariants\,\cite{BHMV2}. (Also compare with the argument in \cite[Section 4.2]{ChMa}). Since  $\{e_{\mathbf{i}}\}$ is a Hermitian basis of $Z_r(T^2)^{\otimes n},$ we have
 \begin{equation*}
 \lambda_{\mathbf{i}}=\langle Z_r'(M),e_{\mathbf{i}} \rangle.
 \end{equation*}
 A tubular neighborhood $N_L$ of $L$  is a disjoint union of solid tori  $\underset{k=1}{\overset{n}{\coprod}}T_k.$
 We let $ L(e_{\mathbf{i}})$ be the element of $K_A(N_L)$ obtained by cabling the component of the $L$ in $T_k$ using  the $i_k$-th Chebyshev polynomial $e_{i_k}.$ 
Then in $Z_r'(T)^{\otimes n},$ we have
\begin{equation*}
e_{\mathbf{i}}= Z_r'(N_L, L(e_{\mathbf{i}})). 
\end{equation*}
Now by Theorem \ref{TQFT} (4),  since $S^3=M \cup (-N_L),$ we have
 \begin{equation*}
 \langle Z_r'(M),e_{\mathbf{i}} \rangle= \langle Z_r'(M), Z_r'(N_L, L(e_{\mathbf{i}}))  \rangle=  \langle   M\cup(-N_L ),  L(e_{\mathbf{i}}))  \rangle_{r}'=\langle S^3 , L(e_{\mathbf{i}}) \rangle_r'.
  \end{equation*}
Finally, by Remark \ref{sphere} (2),  we have
$$\langle S^3 , L(e_{\mathbf{i}}) \rangle_r'=\eta_r' \langle e_{\mathbf{i}}\rangle_L.$$
Therefore, we have 
$$\lambda_{\mathbf{i}}=\eta_r' \langle e_{\mathbf{i}}\rangle_L,$$
which finishes the proof in the case of $r=2m+1.$
\\

The argument of the remaining case is very similar.
By Theorem \ref{RTdoubleTV}, we obtain
\begin{equation*}
TV_r(M)=\langle D(M) \rangle_{r}=\langle Z_{r}(M),Z_{r}(M) \rangle.
\end{equation*}
Working with the Hermitian basis
 $\{e_i\}_{i=0,\ldots ,r-2}$  of $Z_{2r}(T^2)$ given in Theorem \ref{TQFTbasis} (1),  we have
 $$TV_r(M)=\underset{0 \leqslant \mathbf{i} \leqslant r-2}{\sum}|\lambda_{\mathbf{i}}|^2,$$
 where
  $\lambda_{\mathbf{i}}=\langle Z_{r}(M),e_{\mathbf{i}}  \rangle$   and  $e_{\mathbf{i}}= Z_{r}(N_L, L(e_{\mathbf{i}})).$
Now by Theorem \ref{TQFT} (4)
 and Remark \ref{sphere},
one sees
$$\lambda_{\mathbf{i}}=\eta_r\langle e_{\mathbf{i}}\rangle_L,$$
which finishes the proof.
\end{proof}


\section{Applications to Conjecture \ref{TVvolumeconj}} \label{Applications}
In this section we use Theorem \ref{TV=CJP} to determine the asymptotic behavior of the Turaev-Viro invariants for some  hyperbolic 
 knot and link complements. In particular, we verify  Conjecture \ref{TVvolumeconj}
 for the complement of the figure-eight knot and the  Borromean rings. To the best of our knowledge these are the first calculations of this kind.

\subsection{The figure-eight complement}

The following theorem verifies Conjecture \ref{TVvolumeconj} for the figure-eight knot.

\begin{theorem}\label{figure8}
Let $K$ be the figure-eight knot and   let $M$ be the complement of $K$ in $S^3.$
Then
$$\lim_{r\rightarrow +\infty}\frac{2\pi}{r}\log TV_r(M, e^{\frac{2\pi i}{r}})=\lim_{m\rightarrow +\infty}\frac{4\pi}{2m+1}\log|J_{K, m}(e^{\frac{4\pi i}{2m+1}})|=Vol(M),$$
where $r=2m+1$ runs over all odd integers.  

\end{theorem}

\begin{proof} By Theorem \ref{TV=CJP}, and for odd $r=2m+1$, we have that
$$TV_r(S^3\setminus K,e^{\frac{2\pi i}{r}})=(\eta_r')^2\sum_{i=1}^{m}|J_i(K,t)|^2,$$ where $t=q^2=e^{\frac{4\pi i}{r}}.$
Notice that $(\eta_r')^2$ grows only polynomially in $r.$ 

By Habiro and Le's formula \cite{Habiro}, we have 
$$J_{K,i}(t)=1+\sum_{j=1}^{i-1}\prod_{k=1}^j\Big(t^{\frac{i-k}{2}}-t^{-\frac{i-k}{2}}\Big)\Big(t^{\frac{i+k}{2}}-t^{-\frac{i+k}{2}}\Big),$$
where $t=A^4=e^{\frac{4\pi i}{r}}.$

For each $i$ define the function $g_i(j)$ by
\begin{equation*}
\begin{split}
g_i(j)=&\prod_{k=1}^j\Big|\Big(t^{\frac{i-k}{2}}-t^{-\frac{i-k}{2}}\Big)\Big(t^{\frac{i+k}{2}}-t^{-\frac{i+k}{2}}\Big)\Big|\\
=&\prod_{k=1}^j4\Big|\sin\frac{2\pi(i-k)}{r}\Big|\Big|\sin\frac{2\pi(i+k)}{r}\Big|.
\end{split}
\end{equation*}
Then 

$$|J_{K,i}(t)|\leqslant 1+\sum_{j=1}^{i-1}g_i(j).$$
Now let $i$ be such that $\frac{i}{r} \to a\in[0,\frac{1}{2}]$ as $r \to \infty.$  For each $i,$ let $j_i\in\{1, \dots, i-1\}$ such that $g_i(j_i)$ achieves the maximum. We have that $\frac{j_i}{r}$ converges to some $j_a\in (0,1/2)$  which varies continuously in $a$ when $a$ is close to $\frac{1}{2}.$ Then
$$\lim_{r \to \infty} \frac{1}{r}\log|J_{K,i}|\leqslant \lim_{r \to \infty} \frac{1}{r}\log\Big(1+\sum_{j=1}^{i-1}g_i(j)\Big)= \lim_{r \to \infty} \frac{1}{r}\log \big(g_i(j_i)\big),$$
where the last term equals 
\begin{equation*}
\begin{split}& \lim_{r \to \infty}  \frac{1}{r}\Big(\sum_{k=1}^{j_i}\log\Big|2\sin\frac{2\pi(i-k)}{r}\Big|+\sum_{k=1}^{j_i}\log\Big|2\sin\frac{2\pi(i+k)}{r}\Big|\Big)\\
=& \frac{1}{2\pi}\int_0^{j_a\pi}\log\Big(2|\sin\big(2\pi a-t\big)|\Big)dt+ \frac{1}{2\pi}\int_0^{j_a\pi}\log\Big(2|\sin\big(2\pi a+t\big)|\Big)dt\\
=& -\frac{1}{2\pi}\Big(\Lambda\big(2\pi (j_a- a)\big)+\Lambda\big(2\pi a\big)\Big) -\frac{1}{2\pi}\Big(\Lambda\big(2\pi (j_a+ a)\big)-\Lambda\big(2\pi a\big)\Big)\\
=& -\frac{1}{2\pi}\Big(\Lambda\big(2\pi (j_a- a)\big)+\Lambda\big(2\pi (j_a+ a)\big)\Big).
\end{split}
\end{equation*}
Here $\Lambda$ denotes the  Lobachevsky function.
Since $\Lambda(x)$ is an odd function and achieves the maximum at $\frac{\pi}{6},$ the last term above is less than or equal to $$\frac{\Lambda(\frac{\pi}{6})}{\pi}=\frac{3\Lambda(\frac{\pi}{3})}{2\pi}=\frac{Vol(S^3\setminus K)}{4\pi}.$$
We also notice that for $i=m,$ $\frac{i}{r} = \frac{m}{2m+1} \to \frac{1}{2},$ $j_{\frac{1}{2}}=\frac{5}{12}$ and all the inequalities above become equalities. Therefore, the term $|J_{K,m}(t)|^2$ grows the fastest, and 
$$\lim_{r\rightarrow +\infty}\frac{2\pi}{r}\log TV_r(S^3\setminus K,A^2)=\lim_{r\rightarrow +\infty}\frac{2\pi}{r}\log|J_{K,m}(t)|^2=Vol(S^3\setminus K).$$
\end{proof}


\subsection{The Borromean rings complement} 

In this subsection we prove the following theorem that verifies Conjecture \ref{TVvolumeconj} for the 3-component   Borromean rings.
\begin{theorem}\label{borromean} Let $L$ be the 3-component  Borromean rings, and   let $M$ be the complement of $L$ in $S^3.$
Then
$$\lim_{r\rightarrow +\infty}\frac{2\pi}{r}\log TV_r(M, e^{\frac{2\pi i}{r}})=\lim_{m\rightarrow +\infty}\frac{4\pi}{2m+1}\log|J_{L, m}(e^{\frac{4\pi i}{2m+1}})|=Vol(M),$$
where $r=2m+1$ runs over all odd integers.  Here,  $J_{L, m}$ denotes the colored Jones polynomial where all the components of $L$  are colored by $m.$  
\end{theorem}

The proof  relies on the following formula for the colored Jones polynomials of the Borromean rings given by Habiro\,\cite{Habiro, Habiro1}. Let $L$ be the Borromean rings and $k,$ $l$ and $n$ be non-negative integers. Then
\begin{equation}
\label{habiro}
 J_{L,(k,l,n)}(t)=\underset{j=0}{\overset{\mathrm{min}(k,l,n)-1}{\sum}} (-1)^j \frac{[k+j]![l+j]![n+j]!}{[k-j-1]![l-j-1]![n-j-1]!}\left( \frac{[j]!}{[2j+1]!}\right)^2. \end{equation}
 Recall that in this formula $[n]=\frac{t^{n/2}-t^{-n/2}}{t^{1/2}-t^{-1/2}}$ and $[n]!=[n][n-1]\ldots [1]$. 
 From now on we specialize at $t=e^{\frac{4\pi i}{r}}$ where $r=2m+1.$  
 We have
$$[n]=\frac{2\sin(\frac{2n\pi}{r})}{2\sin(\frac{2\pi}{r})}=\frac{\lbrace n \rbrace}{\lbrace 1 \rbrace},$$
where we write
$\lbrace j \rbrace =2\sin(\frac{2j\pi}{r}).$
 We can rewrite  formula (\ref{habiro}) as 
$$J_{L,(k,l,n)}(e^{\frac{4i\pi}{r}})=\underset{j=0}{\overset{\mathrm{min}(k,l,n)-1}{\sum}} (-1)^j\frac{ 1}{\lbrace 1 \rbrace} \frac{\lbrace k+j \rbrace!\lbrace l+j \rbrace!\lbrace n+j\rbrace!}{\lbrace k-j-1\rbrace!\lbrace l-j-1\rbrace!\lbrace n-j-1\rbrace!}\left( \frac{\lbrace j\rbrace !}{\lbrace 2j+1\rbrace !}\right)^2.$$

Next we establish three lemmas needed for the proof of Theorem \ref{borromean}.

\begin{lemma}\label{quantumint}
For any integer $j$ with $0<j<r,$ we have
$$\log(|\lb j \rb!|)=-\frac{r}{2\pi}\Lambda(\frac{2j\pi}{r})+O(\log (r)),$$
where $O(\log(r))$ is uniform: there is a constant $C$ independent of $j$ and  $r$, such that $O(\log r)\leq C \log r.$
\end{lemma}
\begin{proof}
This result is an adaptation of the result in \cite{gale} for r odd. 
By the Euler-Mac Laurin summation formula, for any twice differentiable function $f$ on $[a,b]$ where $a$ and $b$ are integer, we have
$$\underset{k=a}{\overset{b}{\sum}} f(k)=\int_a^b f(t)dt +\frac{1}{2}f(a)+\frac{1}{2}f(b)+R(a,b,f),$$
where $$|R(a,b,f)|\leqslant \frac{3}{24} \int_a^b |f''(t)|dt.$$ Applying this to 
$$\log(|\lb j \rb!|)=\underset{k=1}{\overset{j}{\sum}} \log\left(2|\sin(\frac{2k\pi}{r})|\right),$$
we get
$$\begin{array}{rcl}

\log(|\lb j \rb!|)&=&\int_{1}^{j}\log(2|\sin(\frac{2t\pi}{r})+\frac{1}{2}(f(1)+f(j))+R(\frac{2\pi}{r},\frac{2j\pi}{r},f)
\\ &=&\frac{r}{2\pi}\int_{\frac{2\pi}{r}}^{\frac{2j\pi}{r}}\log(2|\sin(\frac{2t\pi}{r})+\frac{1}{2}(f(1)+f(j))+R(\frac{2\pi}{r},\frac{2j\pi}{r},f)
\\&=&\frac{r}{2\pi}(-\Lambda(\frac{2j\pi}{r})+\Lambda(\frac{2\pi}{r}))+\frac{1}{2}(f(1)+f(j))+R(\frac{2\pi}{r},\frac{2j\pi}{r},f),
\end{array}$$
where $f(t)=\log(2|\sin(\frac{2t\pi}{r})|).$

Since we have $|r\Lambda(\frac{2\pi}{r})|\leqslant C' \log(r)$ and $|f(1)+f(j)|\leqslant C''\log(r)$ for constants $C'$ and $C''$ independent of $j,$ and since
$$R(1,j,f)=\int_1^j |f''(t)|dt=\int_1^j \frac{4\pi^2}{r^2}\frac{1}{\sin(\frac{2\pi t}{r})^2}=\frac{2\pi}{r}\left(\cot (\frac{2j\pi}{r})-\cot(\frac{2\pi}{r})\right)=O(1),$$
we get $$\log(|\lb j \rb!|)=-\frac{r}{2\pi}\Lambda(\frac{2j\pi}{r})+O(\log (r))$$ as claimed.
\end{proof}

Lemma \ref{quantumint} allows us to get an estimation of terms that appear in Habiro's sum for the multi-bracket of Borromean rings. We find that
\begin{multline*}\log\left| \frac{1}{\lbrace 1 \rbrace} \frac{\lbrace k+j \rbrace!\lbrace l+j \rbrace!\lbrace n+j\rbrace!}{\lbrace k-i-1\rbrace!\lbrace l-i-1\rbrace!\lbrace n-i-1\rbrace!}\left( \frac{\lbrace i\rbrace !}{\lbrace 2i+1\rbrace !}\right)^2 \right| \\ =-\frac{r}{2\pi}(f(\alpha,\theta)+f(\beta,\theta)+f(\gamma,\theta))+O(\log(r)),
\end{multline*}
where $\alpha=\frac{2k\pi}{r},$ $\beta=\frac{2l\pi}{r},$ $\gamma=\frac{2n\pi}{r}$ and  $\theta=\frac{2j\pi}{r},$ and 
$$f(\alpha,\theta)=\Lambda(\alpha+\theta)-\Lambda(\alpha-\theta)+\frac{2}{3}\Lambda(\theta)-\frac{2}{3}\Lambda(2\theta).$$

\begin{lemma}\label{upperboundterms}The minimum of the function $f(\alpha,\theta)$ is $-\frac{8}{3}\Lambda(\frac{\pi}{4})=-\frac{v_8}{3}.$ This minimum is attained for $\alpha=0$ modulo $\pi$ and $\theta=\frac{3\pi}{4}$ modulo $\pi.$
\end{lemma}
\begin{proof}
The critical points of $f$ are given by the conditions
$$ \begin{cases} 
\Lambda'(\alpha+\theta)-\Lambda'(\alpha-\theta)=0 
\\ \Lambda'(\alpha+\theta)+\Lambda'(\alpha-\theta)+\frac{2}{3}\Lambda'(\theta)-\frac{4}{3}\Lambda'(2 \theta)=0.
\end{cases}$$
As $\Lambda'(x)=2\log|\sin(x)|,$ the first condition is equivalent to $\alpha+\theta=\pm \alpha-\theta \mod \pi.$ Thus, either $\theta=0 \mod \frac{\pi}{2}$ in which case $f(\alpha,\theta)=0,$ or $\alpha=0 \ \textrm{or} \ \frac{\pi}{2} \mod \pi.$

In the second case, as the Lobachevsky function has the symmetries $\Lambda(-\theta)=-\Lambda(\theta)$ and $\Lambda(\theta +\frac{\pi}{2})=\frac{1}{2}\Lambda(2\theta)-\Lambda(\theta),$ we get 
$$f(0,\theta)=\frac{8}{3}\Lambda(\theta)-\frac{2}{3}\Lambda(2\theta),$$ 
and 
$$f(\frac{\pi}{2},\theta)=\frac{1}{3}\Lambda(2\theta)-\frac{4}{3}\Lambda(\theta).$$
 We get critical points when $2\Lambda'(\theta)=\Lambda(2\theta)$ which is equivalent to $(2 \sin (\theta))^2=2|\sin (2\theta)|.$ This happens only for $\theta=\frac{\pi}{4}\ \textrm{or} \ \frac{3\pi}{4} \mod \pi$ and the minimum value is $-\frac{8}{3}\Lambda(\frac{\pi}{4}),$ which is obtained only for $\alpha=0 \mod \pi$ and $\theta=\frac{3\pi}{4} \mod \pi$ only. 
\end{proof}
\begin{lemma}\label{maxterm} If $r=2m+1,$ we have that 
$$\log(|J_{L,(m,m,m)}(e^{\frac{4i\pi}{r}})|)=\frac{r}{2\pi}v_8 +O(\log(r)).$$
\end{lemma}
\begin{proof}
Again, the argument is very similar to the argument of the usual volume conjecture for the Borromean ring in Theorem A.1 of  \cite{gale}.
We remark that quantum integer $\lb n \rb$ admit the symmetry that $$\lb m+1+i \rb=-\lb m-i \rb$$ for any integer $i.$

 Now, for $k=l=n=m,$ Habiro's formula for the colored Jones polynomials turns into
\begin{equation*}
\begin{split}
J_{L,(m,m,m)}(t)&=\underset{j=0}{\overset{m-1}{\sum}} (-1)^j\frac{\lb m \rb^3}{\lb 1 \rb } \left( \underset{k=1}{\overset{j}{\prod}}\lb m+k \rb \lb m-k \rb \right)^3 \left( \frac{\lbrace j\rbrace !}{\lbrace 2j+1\rbrace !}\right)^2\\
&=\underset{j=0}{\overset{m-1}{\sum}} \frac{\lb m \rb^3 \lb m+j+1 \rb}{\lb 1 \rb \lb m+1 \rb } \left( \underset{k=1}{\overset{j}{\prod}}\lb m+k \rb  \right)^6 \left( \frac{\lbrace j\rbrace !}{\lbrace 2j+1\rbrace !}\right)^2.
\end{split}
\end{equation*}
Note that as $\lb n \rb=\sin (\frac{2 \pi n}{2m+1})<0$ for $n \in \lb m+1,m+2, \ldots,2m \rb,$  the factor $\lb m+j+1 \rb$ will always be negative for $0 \leqslant m-1.$ Thus all terms in the sum have the same sign. Moreover, there is only a polynomial in $r$ number of terms in the sum as $m=\frac{r-1}{2}.$  Therefore, $\log(|J_{L,(m,m,m)}|)$ is up to $O(\log (r))$ equal to the logarithm of the biggest term. But the term $j=\lfloor \frac{3r}{8}\rfloor$ corresponds to $\alpha=\frac{2(m-1)\pi}{r}=0+O(\frac{1}{r}) \mod \pi$ and $\theta =\frac{2 j\pi}{r}=\frac{3\pi}{4}+O(\frac{1}{r}) \mod \pi,$ so 
$$\log \left|\frac{\lb m \rb^3}{\lb 1 \rb } \left( \underset{k=1}{\overset{m-1}{\prod}}\lb m+k \rb \lb m-k \rb \right)^3 \left( \frac{\lbrace m-1\rbrace !}{\lbrace 2m-1\rbrace !}\right)^2\right|=\frac{r}{2\pi}v_8+O(\log(r)),$$
and
$$\frac{2\pi}{r}\log |J_{L,(m,m,m)}|=v_8+O(\frac{\log(r)}{r}).$$
\end{proof}

\begin{proof}[Proof of Theorem  \ref{borromean}] By Theorem \ref{TV=CJP}, we have 
$$TV_r'(S^3\setminus L,e^{\frac{2\pi i}{r}})=(\eta_{r}')^{2}\underset{1 \leqslant k,l,n\leqslant m}{\sum}|J_{L,(k,l,n)})(e^{\frac{4i\pi}{r}})|^2.$$
This is a sum of $m^3=\left(\frac{r-1}{2}\right)^3$ terms, the  logarithm of all of which are less than $\frac{r}{2\pi}(2v_8)+O(\log(r))$ by Lemma \ref{upperboundterms}.  Also, the term $|J_{L,(m,m,m)}(e^{\frac{4i\pi}{r}})|^2$ has logarithm $\frac{r}{2\pi}(2v_8)+O(\log(r)).$
 Thus we have 
$$\underset{r\rightarrow \infty}{\lim} \frac{2\pi}{r}\log (TV_r'(S^3\setminus L),e^{\frac{2\pi i}{r}})=2v_8=\vol (S^3\setminus L).$$
\end{proof}

Finally we note that Theorem \ref{conjecturever} stated in the introduction follows by Theorems \ref{figure8} and \ref{borromean}.


\section{Turaev-Viro Invariants and simplicial volume} \label{general}

Given a link $L$ in $S^3$, there is a unique up to isotopy collection $T$ of essential embedded tori in $M=S^3\setminus L$ so that each component of $M$ cut along $T$ is either hyperbolic or a Seifert fibered space.
This is the toroidal or JSJ-decomposition of $M$ \cite{jaco-shalen}.
Recall that the simplicial volume (or Gromov norm) of $L,$ denoted by $||L||,$  is the sum  of the volumes of the hyperbolic pieces of the decomposition, divided by $v_3;$ the volume of the regular ideal tetrahedron in the hyperbolic space.
In particular,  if the toroidal decomposition has  no hyperbolic pieces, then we have $||L||=0.$  It is known \cite{Soma}  that the simplicial volume is additive under  split unions and connected summations of links.
That is, we have
$$ ||L_1 \sqcup L_2||=  ||L_1 \# L_2||= ||L_1||+ ||L_2||.$$
We note that the connected sum for multi-component  links is not uniquely defined, it depends on the components of links being connected. 

\begin{conjecture}  \label{simplicial} For every link $L\subset S^3,$ we have

$$\lim_{r\to \infty} \frac {2\pi}{r}\log (TV_r(S^3\setminus L, e^{\frac{2\pi i}{r}})) = v_3 ||L||,$$
where $r$ runs over all odd integers.
\end{conjecture}

Theorem \ref{TV=CJP}  suggests that the Turaev-Viro invariants are a  better object to study for the volume conjecture for links.
As remarked in \cite{murakami:vol-conj},  all the Kashaev invariants of a split link are zero. As a result, the original simplicial volume conjecture\,\cite{murakami:vol-conj}
is not true for split links. On the other hand, Corollary \ref{positivity} 
implies that  $TV_r'(S^3\setminus L,q)\neq 0$ for any $r\geqslant 3$ and any primitive root of unity $q=A^2.$ 
\\

Define the double of a knot complement to be the double of the complement of a tubular neighborhood of the knot. Then Theorem \ref{RTdoubleTV} and the main result of \cite{Roberts} implies that if Conjecture \ref{simplicial} holds for a link, then it holds for the double of its complement. In particular, by Theorem \ref{conjecturever},  we have 

\begin{corollary} Conjecture \ref{simplicial} is true for the double of the figure-eight and the Borromean rings complement.
\end{corollary}

Since colored Jones polynomials are multiplicative under split union of links, Theorem \ref{TV=CJP} also implies that $ TV_r'(S^3\setminus L,q)$ is up to a factor multiplicative under split union.

\begin{corollary} For any odd integer $r\geqslant 3$ and $q=A^2$ for a primitive $2r$-th root of unity $A,$ 
$$TV_r'(S^3\setminus (L_1 \sqcup L_2),q)=  (\eta_{r}')^{-1}  TV_r'(S^3\setminus L_1,q)\cdot  TV_r'(S^3\setminus L_2,q).$$
\end{corollary}

The additivity of simplicial volume implies that if Conjecture \ref{simplicial} is true for $L_1$ and $L_2,$
then it is true for the split union $L_1\sqcup L_2.$ 
\\

Next we discuss the behavior of the Turaev-Viro invariants under taking connected sums of links.   With our normalization of the colored Jones polynomials,  we have  that $$J_{L_1 \# L_2,\mathbf i }(t)= [i] J_{L_1 ,\mathbf i_1} (t)\cdot J_{L_2,\mathbf i_2}(t),$$ where $\mathbf i_1$ and $\mathbf i_2$ are respectively  the restriction of $\mathbf i$ to $L_1$ and $L_2,$ and $i$ is the component of $\mathbf i$ corresponding to the component of $L_1\#L_2$ coming from the connected summation. This implies the following.

\begin{corollary}\label{cs}  Let $A$ be a primitive $2r$-th root of unity.  For any odd integer $r\geqslant 3,$ $q=A^2$ and $t=A^4$, we have
$$TV_r'(S^3\setminus L_1\#L_2,q)=(\eta_{r}')^{2}\sum_{1\leqslant \mathbf i\leqslant m}[i]^2|J_{L_1,\mathbf i_1}(t)|^2|J_{L_2,\mathbf i_2}(t)|^2,$$
where $\mathbf i_1$ and $\mathbf i_2$ are respectively  the restriction of $\mathbf i$ to $L_1$ and $L_2,$ and $i$ is the component of $\mathbf i$ corresponding to the component of $L_1\#L_2$ coming from the connected summation.
\end{corollary}

In the rest of this section, we focus on the value $q=e^{\frac{2 i \pi}{r}}$ for odd $r=2m +1.$
Notice that in this case,  the quantum integers $[i]$ for $1\leqslant i \leqslant m$ are non-zero and their logarithms  are of order $O(\log r).$
Corollary \ref{cs} implies that

\begin{equation*} 
\begin{split}
\underset{r \rightarrow \infty}{\limsup} \frac{2\pi}{r}\log TV_r'&(S^3\setminus L_1\# L_2,q)\\&\leqslant  \underset{r \rightarrow \infty}{\limsup} \frac{2\pi}{r}\log TV_r'(S^3\setminus L_1 ,q)+\underset{r \rightarrow \infty}{\limsup} \frac{2\pi}{r}\log TV_r'(S^3\setminus L_2, q). 
\end{split}
\end{equation*}

Moreover if we  assume a positive answer to Question  \ref{dominanttermconj} for $L_1$ and $L_2,$ then the  term $|J_{L_1\#L_2,m}(t)|^2$ of the sum for $L_1\#L_2$ satisfies 
$$\underset{r \rightarrow \infty}{\lim} \frac{2\pi}{r}\log |J_{L_1\#L_2,m}(t)|^2=\vol (S^3\setminus L_1\#L_2).$$
It  follows that if the answer to Question \ref{dominanttermconj} is positive, and Conjecture \ref{simplicial} is true for links  $L_1$ and $L_2,$ then Conjecture \ref{simplicial} is true for their connected sum. In particular, Theorem \ref{conjecturever} implies the following.

\begin{corollary} 
 Conjecture \ref{simplicial} is true for any 
 link obtained by connected sum of the figure-eight and the Borromean rings.
\end{corollary}

We finish the section with the proof of Theorem \ref{zero}, verifying Conjecture \ref{simplicial} for knots of simplicial volume zero.

\begin{named}{Theorem \ref{zero}} Let $K\subset S^3$ be knot with simplicial volume  zero. Then, we have

$$\lim_{r\to \infty} \frac {2\pi}{r} {\log (TV_r(S^3\setminus K, e^{\frac{2\pi i}{r}}))}  =||K||=0,$$
where $r$ runs over all odd integers.
\end{named} 

\begin{proof} By part (2) of Theorem \ref{TV=CJP}  we have

\begin{equation}
TV_r(S^3\setminus K, e^{\frac{2i\pi}{r}})=(\eta_{r}')^{2}\underset{1 \leqslant i\leqslant m}{\sum}
|  J_{L, {i}} (e^{\frac{4i\pi}{r}})|^2.
\label{formula}
\end{equation}
Since 
$J_{K,1}(t)=1,$ we have
 $TV_r(S^3\setminus K)\geqslant \eta_r'^2>0$ for any knot $K.$
 Thus for $r> > 0$ the sum of the values of the colored Jones polynomials in (\ref{formula})  is larger or equal to $1.$
 On the other hand, we have $\eta_{r}'\neq 0$  and 
${\displaystyle {\frac{\log (|\eta_{r}'|^2)}{r}\rightarrow 0}}$ as $r \rightarrow \infty.$
Therefore, 
$$\underset{r \rightarrow \infty}{\liminf} \frac{\log|TV_r(S^3 \setminus K)|}{r}\geqslant 0.$$

Now we only need to prove that for simplicial volume zero knots, we have  $$\underset{r \rightarrow \infty}{\limsup}\frac{\log|TV_r(S^3 \setminus K)|}{r}\leqslant 0.$$
By Theorem \ref{TV=CJP}, part (2) again, it suffices to prove that the $L^1$-norm $||J_{K,i}(t)||$ of the  colored Jones polynomials of any knot $K$ of simplicial volume zero is bounded by a polynomial in $i.$ By Gordon\,\cite{Gordon}, the set of  knots of simplicial volume zero is generated by torus knots, and is closed under taking connected sums and cablings. Therefore, it suffices to prove that the set of knots whose colored Jones polynomials have  $L^1$-norm growing at most polynomially contains the torus knots, and is closed under taking connected sums and cablings.

From Morton's formula\,\cite{Morton}, for the torus knot $T_{p,q},$ we have
$$J_{T_{p,q},i}(t)=t^{pq(1-i^2)}\underset{|k|=-\frac{i-1}{2}}{\overset{\frac{i-1}{2}}{\sum}}\frac{t^{4pqk^2-4(p+q)k+2}-t^{4pqk^2-4(p-q)k-2}}{t^2-t^{-2}}.$$
Each fraction in the summation can be simplified to a geometric sum of powers of $t^2,$ and hence has $L^1$-norm less than $2qi+1.$ From this we have $||J_{T_{p,q},i}(t)||=O(i^2).$

For a connected sum of knots, we recall that the $L^1$-norm of a Laurent polynomial is
$$||\underset{d \in \Z}{\sum} a_d t^d||=\underset{d \in \Z}{\sum}|a_d|.$$
For a Laurent polynomial $R(t)=\underset{f \in \Z}{\sum} c_f t^f,$ we let
$$deg(R(t))=max (\lbrace d \ / c_d\neq 0\rbrace)-min ( \lbrace d \ / c_d\neq 0\rbrace).$$
Then for two Laurent polynomials $P(t)=\underset{d \in \Z}{\sum} a_d t^d$ and $Q(t)=\underset{e \in \Z}{\sum} b_e t^e,$ we have
\begin{eqnarray*}
||PQ||=||\left(\underset{d \in \Z}{\sum}a_d t^d\right)\left(\underset{d \in \Z}{\sum}b_d t^d\right)||&\leqslant& ||\underset{f \in \Z}{\sum}\ \left(\underset{d+e=f}{\sum}a_d b_e\right)t^f||
\\ &\leqslant & \mathrm{deg}(PQ) \underset{d+e=f}{\sum}|a_d b_e|
\\ &\leqslant & \mathrm{deg}(PQ)||P|| ||Q||.
\end{eqnarray*}
Since the  $L^1$-norm  of $[i]$ grows polynomially in $i,$ if the   $L^1$-norms of $J_{K_1,i}(t)$ and $J_{K_2,i}(t)$ grow polynomially, then so does that of  $J_{K_1\#K_2,i} (t)=[i] J_{K_1 ,i}(t)\cdot J_{ K_2 ,i}(t).$

Finally, for the $(p,q)$-cabling $K_{p,q}$ of a knot $K,$ the cabling formula\,\cite{MortonS, Roland} says
\begin{equation*}J_{K_{p,q},i}(t)=t^{pq(i^2-1)/4}\underset{k=-\frac{i-1}{2}}{\overset{\frac{i-1}{2}}{\sum}}t^{-pk(qk+1)}J_{K,2qn+1}(t),
\label{cabling}
\end{equation*}
where $k$ runs over integers if $i$ is odd and over half-integers if $i$ is even.
It  implies that if $||J_{K,i}(t)||=O(i^d),$ then $||J_{K_{p,q},i}(t)||=O(i^{d+1}).$
\end{proof}

By Theorem  \ref{TV=CJP} and the argument in the beginning of the proof of Theorem \ref{zero} applied to links we obtain the following.

\begin{corollary} \label{lower} For every link $L\subset S^3,$ we have
$$\underset{r \rightarrow \infty}{\liminf} \frac{\log|TV_r(S^3 \setminus L)|}{r}\geqslant 0, $$
where $r$ runs over all odd integers.
\end{corollary}
 As said earlier, there is no lower bound for the growth rate of the Kashaev invariants that holds for all links; and no such bound is known for knots as well.


\appendix
\section{The relationship between $TV_r(M)$ and $TV_r'(M)$}\label{A}

The goal of this appendix is to prove Theorem \ref{TVTV'}. To this end, it will be convenient  to modify the definition
of the Turaev-Viro invariants given in Subsection 2 and use the formalism of quantum 6$j$-symbols as in \cite{TuraevViro}.

For $i\in I_r,$ we  let $$|i|=(-1)^i[i+1],$$
and for each admissible triple $(i,j,k),$ we let
$$\big|i, j, k\big| =(-1)^{-\frac{i+j+k}{2}} \frac{[\frac{i+j-k}{2}]![\frac{j+k-i}{2}]![\frac{k+i-j}{2}]!}{[\frac{i+j+k}{2}+1]!}.$$
Also for each admissible $6$-tuple $(i,j,k,l,m,n),$ we let 
\begin{equation*}
\begin{split}
&\bigg|\begin{array}{ccc}i & j & k \\l & m & n \\\end{array} \bigg|=\sum_{z=\max \{T_1, T_2, T_3, T_4\}}^{\min\{ Q_1,Q_2,Q_3\}}\frac{(-1)^z[z+1]!}{[z-T_1]![z-T_2]![z-T_3]![z-T_4]![Q_1-z]![Q_2-z]![Q_3-z]!}.
\end{split}
\end{equation*}

Consider a triangulation $\mathcal T$ of $M,$ and let $c$ be an admissible coloring of $(M,\mathcal T)$ at level $r.$ For each edge $e$ of $\mathcal T,$ 
 we let 
$$|e|_c=|c(e)|,$$ 
and for each face $f$ with edges $e_1, e_2$ and $e_3,$ we let 
$$|f|_c=\big|c(e_1), c(e_2), c(e_3)\big|.$$
Also for each tetrahedra $\Delta$  with edges $e_{ij},$ $\{i,j\}\subset\{1,\dots, 4\},$ we let
\begin{equation*}
|\Delta|_c=\bigg|
      \begin{array}{ccc}
        c(e_{12}) & c(e_{13}) & c(e_{23}) \\
       c(e_{34}) & c(e_{24}) & c(e_{14}) \\
      \end{array} \bigg|.
\end{equation*}

Now recall the invariants $TV_r(M)$ and  $TV_r'(M)$ given in Definitions \ref{TV} and \ref{TV'}, respectively.
Then we have the following.

\begin{proposition} \label{altdef}
\begin{enumerate}[(a)]
\item For any integer $r\geqslant 3,$
 $$TV_r(M)= \eta_r^{2|V|}\sum_{c\in A_r}{\prod_{e\in E}|e|_c{\prod_{f\in E}|f|_c}\prod_{\Delta\in T}|\Delta|_c}.$$
\item For any odd integer $r\geqslant 3,$
 $$TV_r'(M)= (\eta_r')^{2|V|}\sum_{c\in A_r'}{\prod_{e\in E}|e|_c{\prod_{f\in E}|f|_c}\prod_{\Delta\in T}|\Delta|_c}.$$

\end{enumerate}
\end{proposition}
\begin{proof} The proof is  a straightforward calculation. \end{proof}

Next we establish four lemmas on which the  proof  of Theorem \ref{TVTV'}  will rely. We will use the notations $|i|_r,$ $|i,j,k|_r$ and $\bigg|\begin{array}{ccc}i & j & k \\l & m & n \\\end{array} \bigg|_r$ respectively to mean the values of $|i|,$ $|i,j,k|$ and $\bigg|\begin{array}{ccc}i & j & k \\l & m & n \\\end{array} \bigg|$ at a primitive $2r$-th root of unity $A.$

\begin{lemma}\label{A1}$|0|_3= |1|_3=1,$ $\big|0, 0, 0\big|_3=\big|1, 1, 0\big|_3=1$ and
\begin{equation*}
\ \bigg|\begin{array}{ccc}0 & 0 & 0 \\0 & 0 & 0 \\\end{array} \bigg|_3
= \bigg|\begin{array}{ccc} 0 & 0 & 0 \\1 & 1 & 1 \\\end{array} \bigg|_3=\bigg|\begin{array}{ccc}1 &1 & 0 \\1 & 1 & 0 \\\end{array} \bigg|_3=1.
\end{equation*}
\end{lemma}

\begin{proof} A direct calculation.
\end{proof}

The following lemma can be considered as a Turaev-Viro setting  analogue  of Theorem \ref{TQFTbasis} (3).
\begin{lemma}For $i\in I_r,$ let $i'=r-2-i.$
 \begin{enumerate}[(a)] 
 \item If $i\in I_r,$ then $i'\in I_r.$ Moreover, $|i'|_r=|i|_r.$
 \item If the triple $(i,j,k)$ is admissible, then so is the triple $(i', j', k).$ Moreover, $$\big|i',j',k\big|_r=\big|i,j,k\big|_r.$$
\item If the $6$-tuple $(i,j,k,l,m,n)$ is admissible, then so are the 6-tuples $(i,j,k, l',m',n')$ and $(i',j', k, l', m', n).$ Moreover, \begin{equation*}
\bigg|\begin{array}{ccc}i & j & k \\l' & m' & n' \\\end{array} \bigg|_r
=\bigg|\begin{array}{ccc}i & j & k \\l & m & n \\\end{array} \bigg|_r
\text{\quad and\quad }
\bigg|\begin{array}{ccc}i' & j' & k \\l' & m' & n \\\end{array} \bigg|_r
=\bigg|\begin{array}{ccc}i & j & k \\l & m & n \\\end{array} \bigg|_r.
\end{equation*}
\end{enumerate}
\end{lemma}

\begin{proof}  Parts (a) (b) follow easily from the definitions. 

To see the first identity of (c), let $T_i'$ and $Q_j'$ be the  sums for $(i,j,k,l',m',n'),$ involved in the expression of the corresponding $6j$-symbol. 
Namely, let
 $$T'_1=\frac{i+j+k}{2}=T_1,\  \ T'_2=\frac{j+l'+n'}{2}\ \  {\rm and\ }  Q'_2=\frac{i+k+l'+n'}{2}, \ \  {\rm etc}.$$  For the terms in the summations defining the two $6j$-symbols, let us leave $T_1$ alone for now, and consider the other $T_i$'s and $Q_j$'s. Without loss of generality  we assume that, $Q_3\geqslant Q_2\geqslant Q_1 \geqslant T_4\geqslant T_3\geqslant T_2.$ One can easily check that
\begin{enumerate}[(1)] 

\item  $Q_3-Q_1=T'_4-T'_2,$ $Q_2-Q_1=T'_4-T'_3,$ $Q_1-T_4=Q'_1-T'_4,$ $T_4-T_3=Q'_2-Q'_1$ and $T_4-T_2=Q'_3-Q'_1,$ which implies 
\item $Q'_3\geqslant Q'_2 \geqslant Q'_1\geqslant T'_4\geqslant T'_3\geqslant T'_2.$
\end{enumerate} 

For $z$ in between $\max\{T_1,\dots , T_4\}$ and $\min\{Q_1, Q_2, Q_3\},$ let 

$$P(z)=\frac{(-1)^z[z+1]!}{[z-T_1]![z-T_2]![z-T_3]![z-T_4]![Q_1-z]![Q_2-z]![Q_3-z]!},$$ 
and similarly for $z$ in between $\max\{T'_1,\dots , T'_4\}$ and $\min\{Q'_1, Q'_2, Q'_3\}$ let 

$$P'(z)=\frac{(-1)^{z}[z+1]!}{[z-T'_1]![z-T'_2]![z-T'_3]![z-T'_4]![Q'_1-z]![Q'_2-z]![Q'_3-z]!}.$$

Then for any $a\in\{0, 1, \dots, Q_1-T_4=Q'_1-T'_4\}$ one verifies by (1) above that 
\begin{equation}\label{P}
P(T_4+a)=P'(Q'_1-a).
\end{equation}

There are the following three cases to consider. \ \

{\bf Case 1.} $T_1\leqslant T_4$ and $T'_1 \leqslant T'_4.$

 In this case, $T_{\text{max}}=T_4,$ $Q_{\text{min}}=Q_1,$ $T'_{\text{max}}=T'_4$ and $Q'_{\text{min}}=Q'_1.$ By (\ref{P}), we have 
$$\sum_{z= T_4}^{Q_1}P(z)=\sum_{a=0}^{Q_1-T_4}P(T_4+a)=\sum_{a=0}^{Q'_1-T_4}P'(Q'_1-a)=\sum_{z= T'_4}^{Q'_1}P'(z).$$
\ \ 

{\bf Case 2.} $ T_1 > T_4$ but  $T'_1 < T'_4,$ or $ T_1 < T_4$ but  $T'_1 > T'_4.$ 

By symmetry, it suffices to consider the former case. In this case $T_{\text{max}}=T_1,$ $Q_{\text{min}}=Q_1,$ $T'_{\text{max}}=T'_4$ and $Q'_{\text{min}}=Q'_1, $ and $$Q'_1 - (r-2) = \frac{i + j - l -m}{2} = T_1 -T_4.$$ As a consequence $Q'_1>r-2 .$ By (\ref{P}), we have

$$\sum_{z= T_1}^{Q_1}P(z)=\sum_{a=T_1-T_4}^{Q_1-T_4}P(T_4+a)=\sum_{a=Q'_1-(r-2)}^{Q'_1-T'_4}P'(Q'_1-a)=\sum_{z= T'_4}^{r-2}P'(z)=\sum_{z= T'_4}^{Q'_1}P'(z).$$
The last equality is because we have $P'(z)=0,$ for $z> r-2.$

\ \

{\bf Case 3.} $ T_1 > T_4$ and  $T'_1 >T'_4.$ 

In this case we have  $T_{\text{max}}=T_1,$ $Q_{\text{min}}=Q_1,$ $T'_{\text{max}}=T'_1$ and $Q'_{\text{min}}=Q'_1.$ 
We have

 $$Q'_1-(r-2)=\frac{ i+j-l-m}{2} = T_1-T_4>0,$$
  hence $Q_1>r-2.$  Also, we have $$Q'_1-T'_1 = \frac{l'+m'-k }{2}=r-2-T_4.$$ As a consequence, $Q'_1-(r-2)=T'_1-T_4=T_1-T_4>0,$ and hence $Q'_1>r-2.$ By (\ref{P}), we have 

$$\sum_{z= T_1}^{Q_1}P(z)=\sum_{z= T_1}^{r-2}P(z)=\sum_{a=T_1-T_4}^{r-2-T_4}P(T_4+a)=\sum_{a=Q'_1-(r-2)}^{Q'_1-T'_1}P'(Q'_1-a)=\sum_{z= T'_1}^{r-2}P'(z)=\sum_{z= T'_1}^{Q'_1}P'(z).$$
The first and the last equality are because $P(z)=P'(z)=0,$ for $z>r-2.$ 

The second identity of (c) is a consequence of the first.
\end{proof}

As an immediate consequence of the two lemmas above, we have
\begin{lemma}\label{C} 
\begin{enumerate}[(a)]
\item For all $i\in I_r,$ $|i|_r=|0|_3|i|_r$ and $|i'|_r=|1|_3|i|_r.$
\item If the triple $(i,j,k)$ is admissible, then $$\big|i,j,k\big|_r=\big|0,0,0\big|_3\big|i,j,k\big|_r\text{\quad and\quad} \big|i',j',k\big|_r=\big|1,1,0\big|_3\big|i,j,k\big|_r.$$
\item For every admissible $6$-tuple $(i,j,k,l,m,n)$  we have the following.
\begin{equation*}
\begin{split}
\bigg|\begin{array}{ccc}i & j & k \\l & m & n \\\end{array} \bigg|_r&=\ \bigg|\begin{array}{ccc}0 & 0 & 0 \\0 & 0 & 0 \\\end{array} \bigg|_3
\bigg|\begin{array}{ccc}i & j & k \\l & m & n \\\end{array} \bigg|_r,\\
\\
\bigg|\begin{array}{ccc}i & j & k \\l' & m' & n' \\\end{array} \bigg|_r&=\bigg|\begin{array}{ccc} 0 & 0 & 0 \\1 & 1& 1 \\\end{array} \bigg|_3\bigg|\begin{array}{ccc}i & j & k \\l & m & n \\\end{array} \bigg|_r,\\
\\
\bigg|\begin{array}{ccc}i' & j' & k \\l' & m' & n \\\end{array} \bigg|_r&=\bigg|\begin{array}{ccc}1 &1 & 0 \\1 & 1 & 0 \\\end{array} \bigg|_3\bigg|\begin{array}{ccc}i & j & k \\l & m & n \\\end{array} \bigg|_r.
\end{split}
\end{equation*}
\end{enumerate}
\end{lemma} 
\qed
\bigskip

Now we are ready to prove Theorem \ref{TVTV'}.

\begin{proof}[Proof of Theorem \ref{TVTV'}] For (a), we observe that there is a bijection $\phi: I_3\times I'_r \rightarrow I_r$ defined by 
$\phi (0, i)=i$ and $\phi (1, i)=i'.$ This induces a bijection $$\phi: A_3 \times  A'_r \rightarrow  A_r.$$
Then, by Proposition \ref{altdef}, we have 
\begin{equation*}
\begin{split}
TV_3(M)&\cdot TV'_r(M)\\&= \Big(\eta_3^{2|V|}\sum_{c\in A_3}{\prod_{e\in E}|e|_c{\prod_{f\in F}|f|_c}\prod_{\Delta\in T}|\Delta|_c}\Big)\Big(\eta_r'^{2|V|}\sum_{c'\in A'_r}{\prod_{e\in E}|e|_{c'}{\prod_{f\in F}|f|_{c'}}\prod_{\Delta\in T}|\Delta|_{c'}}\Big)\\
&=(\eta_3\eta'_r)^{2|V|}\sum_{(c,c')\in A_3\times A'_r}{\prod_{e\in E} |e|_c|e|_{c'}{\prod_{f\in F} |f|_c|f|_{c'}}\prod_{\Delta\in T} |\Delta|_c|\Delta|_{c'}}\\
&= \eta_r^{2|V|}\sum_{\phi(c,c')\in A_r} {\prod_{e\in E} |e|_{\phi(c,c')}{\prod_{f\in F} |f|_{\phi(c,c')}}\prod_{\Delta\in T} |\Delta|_{\phi(c,c')}}\\&=TV_r(M),
\end{split}
\end{equation*}
where the third equality comes from the fact that  $\eta_r = \eta_3 \cdot\eta_r'$ and Lemma \ref{C}. This finishes the proof of part (a) of the statement of the theorem.
Part  (b) is given in \cite[9.3.A]{TuraevViro}.

To deduce  (c), note that  by Lemma \ref{A1} we have  that
$$TV_3(M)=\sum_{c\in A_3}1=|A_3|.$$ 
Also note that $c\in A_3$ if and only if $c(e_1)+c(e_2)+c(e_3)$ is even for the edges $e_1, e_2, e_3$ of a face. Now consider the handle decomposition of $M$ dual to the ideal triangulation. Then there is a one-to-one correspondence between $3$-colorings and maps
$$\overline c:\{ 2-\text{handles}\} \rightarrow \mathbb Z_2,$$
and $c\in A_3$ if and only if $\overline c$ is a $2$-cycle; that is if and only if ${\overline c}  \in Z_2(M, \mathbb Z_2).$
Hence  we get $|A_3|=\dim{(Z_2(M, \mathbb Z_2)}).$ Since there are no $3$-handles, $H_2(M,\mathbb Z_2)\cong Z_2(M, \mathbb Z_2).$ Therefore, $$TV_3(M)=|A_3|= \dim (H_2(M, \mathbb Z_2))=2^{b_2(M)}.$$
\end{proof}

\bibliographystyle{hamsplain}
\bibliography{biblio}

\noindent 
Effstratia Kalfagianni\\
Department of Mathematics, Michigan State University\\
East Lansing, MI 48824\\
(kalfagia@math.msu.edu)
\\
  
 \noindent 
Renaud Detcherry\\
Department of Mathematics, Michigan State University\\
East Lansing, MI 48824\\
(detcherry@math.msu.edu)
\\

\noindent
Tian Yang\\
Department of Mathematics, Texas A $\&$M University\\
College Station, TX 77843\\
( tianyan@math.tamu.edu)

\end{document}